\documentclass[10pt,twoside, a4paper, english, reqno]{amsart}
\usepackage[dvips]{epsfig}
\usepackage{amscd}
\usepackage{stmaryrd}
\usepackage{amssymb}
\usepackage{amsthm}
\usepackage{amsmath}
\usepackage{graphicx,enumitem,dsfont}
\usepackage{latexsym}

\usepackage{upref}
\usepackage[colorlinks]{hyperref}
\usepackage{color}
\usepackage{subcaption}
\usepackage[usenames,dvipsnames]{xcolor}
\theoremstyle{plain}
\newtheorem{thm}{Theorem}[section]
\theoremstyle{plain}
\newtheorem{lem}[thm]{Lemma}

\theoremstyle{definition}
\newtheorem{defi}{Definition}[section]
\newtheorem{rem}{Remark}[section]

\newtheorem*{maintheorem*}{Main Theorem}

\newenvironment{Assumptions}
{
\setcounter{enumi}{0}

\begin{enumerate}}
{\end{enumerate} }

\newenvironment{Assumptions2}
{
\setcounter{enumi}{0}

\begin{enumerate}}
{\end{enumerate} }

\newcommand{\R}{\ensuremath{\mathbb{R}}}

\newcommand{\goto}{\ensuremath{\rightarrow}}

\newcommand{\eps}{\ensuremath{\varepsilon}}

\numberwithin{equation}{section} \allowdisplaybreaks

\title[Large Deviations Principle]
{Nonlinear stochastic Laplace equation: Large Deviations and Measure Concentration}

\date{}

\author[A. K. Majee]{Ananta K. Majee}
\address[Ananta K. Majee]{\newline
Department of Mathematics,
Indian Institute of Technology Delhi,
Hauz Khas, New Delhi, 110016, India. }
\email[]{majee@maths.iitd.ac.in}


\thanks{}

\numberwithin{equation}{section} \allowdisplaybreaks
\begin{document}
\begin{abstract}
    In this paper, a large deviation 
principle for the strong solution of the $p$-Laplace equation on $\R^d$ driven by small multiplicative Brownian noise is established. The weak convergence approach and the localized time increment estimate plays a crucial role to establish the large deviation principle. Moreover, based on  the Girsanov transformation and the standard $L^2$-uniqueness approach, the quadratic transportation cost information inequality is proved for the strong solution to the underlying problem which then implies the measure concentration phenomenon. 

\end{abstract}
\maketitle
\section {Introduction}
Let $(\Omega, \mathbb{P}, \mathcal{F}, \{ \mathcal{F}_t\}_{t \geq 0})$
be a given filtered probability space satisfying the usual hypothesis. We are interested in the theory of large deviation principle and measure concentration  for the strong solution of the following evolutionary $p$-Laplace equation on the unbounded domain $\R^d$ perturbed by the multiplicative cylindrical Wiener noise:
\begin{equation}\label{eq:spde}
      \begin{aligned}
            du(t,x)  - {\rm div}_x ( |\nabla u|^{p-2}\nabla u) \, dt &= \sigma(u)\,dW(t),~~~(t,x)\in \Pi_T, \\
            u(0,x)=u_0(x),~~x\in \R^d\,,
        \end{aligned}
\end{equation}
where $1<p<\infty,~\Pi_T: = (0,T) \times \R^d$ with $T > 0$ fixed. In \eqref{eq:spde}, $\sigma$ is a given noise coefficient, 
$W$ is a cylindrical Wiener process, with respect to the filtered probability space $(\Omega, \mathbb{P}, \mathcal{F}, \{ \mathcal{F}_t\}_{t \geq 0})$, defined on a given separable Hilbert space $\mathcal{H}$. 
\vspace{0.2cm}

It is well-known that the infiltration of  a fluid, in the presence of  heterogeneous medium or turbulence, is govern by the continuity equation and the generalized $p$-power type version of the Darcy law. Discarding the lower order terms and renormalizing the various physical constants, one model the flow of a fluid through heterogeneous medium or turbulence via the PDEs
\begin{align}
\partial_t u - {\rm div} ( |\nabla u|^{p-2}\nabla u) =f, \label{eq:pde-deterministic}
\end{align}
where $u=u(t,x)$ is the volumetric moisture content, $1\le p<\infty$ varies with the properties of the flow, and $f(t,x)$ represents the external force present in the system; see \cite{Diaz-1994} and references therein.  For fixed $p\in (1,\infty)$, the diffusion operator $\Delta_p u:= {\rm div} ( |\nabla u|^{p-2}\nabla u)$ is known as 
the $p$-Laplace operator. For $p=2$, it is well-known Laplace operator. 
\vspace{0.1cm}

Due to the complexity of the physical phenomena and the multiscale interactions, systems may have some random influences.  It is quite natural to add a stochastic forcing in the governing equation \eqref{eq:pde-deterministic}.  Hence, equation of type \eqref{eq:spde} could be viewed as a stochastic perturbation of an evolutionary $p$-Laplace equation with nonlinear sources--- which appears in many applications in mechanics, physics and biology (non-Newtonian fluids, gas flow in porus media, spread of biological populations etc; see e.g., \cite{aronson-83,caffarelli-87, debenedetto-93,ladyzenskaya-67,wu-2001}). Due to its wide range of applications in physical contexts, it is important to study the well-posedness theory for equation \eqref{eq:spde}.

 \vspace{0.1cm}

In view of non-linearity of diffusion term present in the equation \eqref{eq:spde}, a semigroup representation \cite{prato} of  mild solutions of \eqref{eq:spde} is not possible. In the case of bounded domain, equation \eqref{eq:spde} becomes a monotone nonlinear SPDE driven by cylindrical Wiener noise, and one could apply the method of monotonicity, first developed in \cite{lions-69} for the deterministic equation, on an appropriate choice of the Gelfand triplet $({\tt V}, {\tt H},{\tt V}^*)$~(i.e., ${\tt V}$ be a separable Banach space with its dual ${\tt V}^*$ such that $V \hookrightarrow {\tt H}$ is continuous and dense and ${\tt H}$ is a Hilbert space with the identification ${\tt H}={\tt H}^*$) to prove its well-posedness; see \cite{barbu-2010,bensoussan-72,krylov-81,rockner-2015,pard}.
We mention the works of \cite{majee-20,guy-19}, where the authors studied the SPDE \eqref{eq:spde}, in a bounded domain, with additional drift term 
${\rm div}_x (F(u))$ for some Lipschitz continuous function $F$. Due to the presence of nonlinear perturbation ${\rm div}_x (F(u))$ of the $p$-Laplace operator~($p>2$), it is not feasible to use the results of monotone or locally monotone SPDEs. However, in \cite{guy-19}, they used the technique of semi-implicit time discretization and the Skorokhod and Prokhorov theorems to establish the existence of a martingale solution.  Moreover, employing the Gy$\Ddot{o}$ngy characterization of the convergence in probability and pathwise uniqueness (which was established by $L^1$- contraction method) of weak solutions, the authors established the well-posedness of the strong solution of the underlying problem.  In \cite{majee-20}, the author used pseudo-monotonicity methods, Aldous
tightness criterion, and the Jakubowski-Skorokhod theorem \cite{jakubowski-98} on a non-metric space to prove the
existence of a martingale solution of the evolutionary $p$-Laplace equation driven by L\'{e}vy noise.  We also refer to \cite{majee-24,guy-21} where the well-posedness theory (in a bounded domain) was established for the solution of the PDE governed by an evolutionary $p$-type growth operator with nonlinear random sources. However, due to the absence of Poincare's inequality and certain Sobolev space embedding (for example $W^{1,p}(\R^d)$ is not continuously embedded in $L^2(\R^d)$ for $p> \max\{ d,2\}$ and $1<p < \frac{2d}{d+2}$; see \cite{brezis-2011}), the coercivity condition of the $p$-Laplace operator $-\Delta_p$ in the whole domain $\R^d$ does not hold true. Hence, one cannot use a standard well-posedness theorems in the classical functional setting for all values of $1<p<\infty$. In \cite{schmitz-23}, the authors established the well-posedness result for \eqref{eq:spde} in the framework of Gelfand triple 
$\mathcal{Y} \hookrightarrow L^2(\R^d) \hookrightarrow \mathcal{Y}^\prime$~(see the notation in Section \ref{sec:technical}), which is independent of Sobolev space embedding and space dimension. They have shown that the classical well-posedness theory for monotone operators as in \cite{barbu-2010} as well as the general framework as in \cite{ren-2007} is not applicable for the SPDE \eqref{eq:spde} in the framework of  Gelfand triple $\mathcal{Y} \hookrightarrow L^2(\R^d) \hookrightarrow \mathcal{Y}^\prime$. 
The authors first established the wellposedness result for \eqref{eq:spde} with additive noise based on the time discretization technique and the monotonicity method, and then used a fixed-point argument to prove the well-posedness result for multiplicative noise. 

\vspace{0.2cm}

It is also crucial to analyze and quantify the probability of rare events in various fields, allowing for an understanding of the limiting behaviour of certain probability models. In particular, one may ask the following natural question: when the stochastic perturbation is significantly small, what is the asymptotic relationship between the solution of \eqref{eq:spde} and the corresponding deterministic equation?  In other words, one needs to study the small noise large deviation principle (LDP in short) for the solution of \eqref{eq:spde}. In mechanics, the LDP is used to analyze the behavior of rate fluctuations, while in dynamics, it helps to quantify certain quantities such as energy.
 The thory of LDP plays an important role in various fields such as thermodynamics, information theory, engineering and statistics; see e.g., \cite{dembo, Stroock-1989, Ellis-1985,kurtz,freidlin, stroock,varadhan-84,varadha1} and reference therein. During the last decade there has been growing interest in the study of LDP for stochastic partial differential equations~(SPDEs) and new results are emerging faster than ever before within different frameworks in the literature. We refer to see  e.g.,\cite{chow-92,wang-2006, rock,wang-2012} where the authors have employed the approximation techniques, the contraction principle and the exponential tightness estimations to get the LDP. However, for many complex models, it is  a daunting task to verify some exponential estimates and the tightness property. In recent years, a weak convergence approach \cite{budhi,dupuis}, based on certain variational representation formulas, for the study of the small noise large deviation principle has received a considerable amount of attention. The main
advantage of the variational approach is that, instead of proving the quite complicated exponential tightness estimates, one needs only the fundamental qualitative understanding of the well-posedness and  the stability results of the underlying problem. Based on weak convergence method, a number of authors have contributed since then, and we mention only few see e.g., 
\cite{dong2,dong1, liu,  MSZ, ren} and references therein. In \cite{ren}, Ren and Zhang have studied Freidlin-Wentzell's large deviation for stochastic evolution equation in the evolution triple. The LDP for stochastic evolution equation involving strongly monotone drift  was studied by Liu in \cite{liu}. The study of LDP was carried out for 3D stochastic primitive equation by Dong, Zhai and Zhang in \cite{dong1}. In \cite{millet-2010}, the authors have studied the LDP for the
Stochastic $2D$ hydrodynamical type systems. The LDP theory for stochastic scalar conservation laws was stidied by Dong et al. in \cite{dong2, dong3} and by Behera et al. in \cite{behera-2024} in the kinetic solution framework. By introducing a different sufficient conditions, the authors in \cite{MSZ} have established the LDP for obstacle problems of quasi-linear SPDEs.
The LDP was studied for Multiscale locally monotone SPDEs by Hong et al. in \cite{hong-2021}, Stochastic geometric wave equation and Stochastic two-dimensional nematic liquid crystal model by Brze\'{z}niak et al. in \cite{goldys-2022, rana-2020}, Two-time-scale stochastic convective Brinkman-Forchheimer equations by Mohan in \cite{manil-2023}, and Stochastic shell model of turbulence by Manna et al. in \cite{manna-2009} via weak convergence method.
Very recently, the authors in \cite{kavin-majee-24} established large deviation principle for the strong solution of an evolutionary nonlinear perturbation ${\rm div}_x F(u)$ of  $p$-Laplace equation~($p>2$) driven by small multiplicative one dimensional Brownian noise in bounded domain. By employing semi-discrete time discretization together with a-priori estimation on some appropriate fractional Sobolev space, the authors established  the LDP theory via the weak convergence method. We also mention the work of \cite{majee-24}, where the author established the LDP for a doubly nonlinear PDE driven by small multiplicative Brownian noise in a bounded domain based on motononicity argument and the weak convergence approach. 
 \vspace{0.2cm}
 
 Another important topics of research is to study the concentration of measure phenomenon. For any metric space  $({\tt X},{\tt d})$ equipped with the Borel $\sigma$-field $\mathcal{B}({\tt X})$, define the set
$$ A_r:=\big\{ x\in {\tt X}:~{\rm dist}(x, A)\le r\big\},~~r>0,~~A\in \mathcal{B}({\tt X}).$$
For any Borel probability measure  $\mathbb{Q}$ on ${\tt X}$, consider 
$$ \alpha(r):= \sup\,\big\{ \mathbb{Q}({\tt X}\setminus A_r):~~A\in \mathcal{B}({\tt X}),~\mathbb{Q}(A)\ge \frac{1}{2}\big\}\,.$$
The {\it concentration of measure phenomenon} is the property that 
$$ \alpha(r)\approx 0,~~\text{when}~~~r>>1.$$
The quality of the concentration of $\mathbb{Q}$ depends on the rate at which $\alpha(r)\rightarrow 0 $ as $r \rightarrow \infty$. The normal concentration and concentration of product measures were studied in \cite{guilli,Marton-1996-2,Marton-1996-1, Tal, Tal2,Ledoux-2001}.  The concentration of measure has a close connections with the well-known functional inequalities, e.g., Poincar{\'e} and logarithmic Sobolev inequalities \cite{Bobkov-1999,Patrick-2009, Roberto-2011,otto}.
It is well known that one can study concentration of measure by establishing Talagrand concentration inequalities---widely known as {\it transportation-cost information inequalities} (TCI inequalities in short). These are inequalities that compare Wasserstein distance with relative entropy. In particular, if $\mathbb{Q}$ satisfies $L^1$-transportation cost information inequality~(cf.~Definition \ref{defi:tci}), then 
$$ \alpha(r) \le e^{-\frac{r^2}{8{\tt C}}}\quad \text{for}~~r\ge r_0=2\sqrt{2{\tt C}\ln 2} $$
for some ${\tt C}>0$; see \cite{Marton-1996-1}. The authors in \cite{otto} have studied the relation between the quadratic TCI inequality and other functional inequalities. Moreover, they have shown that quadratic TCI is deducible from the log-Sobolev inequality. So, a crucial
question would be whether the quadratic TCI is strictly weaker than the log-Sobolev inequality or not. More specifically, it would be interesting to establish the quadratic TCI inequality without the prior-knowledge of the log-Sobolev inequality. For a general theory of TCI inequality, we refer to see \cite{villan} and reference therein.  In the recent past, many authors have devoted themselves to study the quadratic TCI inequality for SPDEs  e.g.,\cite{bouf,sarantsev-2019, shang, zhang}. Very recently, the authors in \cite{kavin-majee-24} have derived the quadratic TCI inequality based on the Girsanov transformation along with the $L^1$-contraction approach in the case of bounded diffusion coefficient in the bounded domain. We verify our concentration result by showing that the laws of the solution satisfies a quadratic TCI inequality. 
 \subsection{Aim and outline of the paper}
 In this paper, we wish to generalize the result of \cite{kavin-majee-24} for unbounded domain as well as for the cylindrical Wiener noise---without the additional term ${\rm div}_x (F(u))$ but for any $1<p <\infty$ i.e., to establish the LDP theory for the solution of \eqref{eq:spde} via the weak convergence approach and the quadratic TCI inequality. Due to the lack of compactness of the embedding on the Gelfand triple $\mathcal{Y} \hookrightarrow L^2(\R^d) \hookrightarrow \mathcal{Y}^\prime$, one cannot apply the same method as in \cite{kavin-majee-24}. However, we overcome this issue by deriving a localized time increment estimate as in \cite{millet-2010,hong-2021} for the solution of the skeleton equation (e.g.,~Lemma \ref{lem:time-increment-un-cond-c2}) and using the equivalent sufficient conditions for LDP as in \cite{MSZ}. More precisely, we proceed as follows:
\vspace{0.2cm}

 \begin{itemize}
\item[a)] To prove the LDP for equation \eqref{eq:spde} on the solution space $\mathcal{Z}= C([0,T];L^2(\R^d))$, one crucial step is to prove the well-posedness result of the skeleton equation \eqref{eq:skeleton}.
\begin{itemize}
\item[i)] We first consider equation \eqref{eq:skeleton} with the control $h\in L^\infty(0,T;\mathcal{H})$. To do so, taking motivation from \cite{schmitz-23}, we first show the well-posedness result of the auxiliary equation \eqref{eq:skeleton-auxi} based on the technique of time discretization. We build a sequence of approximate solutions and use a-priori estimates together with the monotonicity argument to show that the limit of the approximate solutions is indeed a unique solution to the auxiliary equation \eqref{eq:skeleton-auxi}. Moreover using the well-posedness result of the auxiliary equation \eqref{eq:skeleton-auxi} and the fixed point argument, we able to establish the well-posedness result for the skeleton equation \eqref{eq:skeleton} in the case $h\in L^\infty(0,T; \mathcal{H})$; see Subsection \ref{subsec:existence-skeleton-bounded-control}. 
\item[ii)] By using a standard approximations argument and the a-priori estimates, we first deduce that the approximation solutions 
$\{u_n\}$ of equation \eqref{eq:skeleton-n-general} is a Cauchy sequence in $C([0,T]; L^2(\R^d))$--- which then yields that the limit function $\bar{u}_*$ of $\{u_n\}$ is the unique solution of the skeleton equation \eqref{eq:skeleton}. 
\end{itemize}
\item[b)] In view of the non-compactness of the embedding $\mathcal{Y}\hookrightarrow L^2(\R^d) \hookrightarrow \mathcal{Y}^\prime$ and the lack of compactness of the control sequences (as we consider only the weak topology on $S_M$ for each fixed $M<\infty$), we use the equivalent sufficient conditions as given in \cite{MSZ}. Thus, to prove LDP, it is enough to validate the conditions \ref{C1} and \ref{C2}.
\begin{itemize}
\item[i)] Thanks to Girsanov's theorem, uniform moment estimates and the uniqueness of the solutions of \eqref{eq:ldp}, we validate the condition \ref{C1}.
\item[ii)] To validate the condition \ref{C2}, one needs to shown the convergence of 
$\bar{u}_n$ to $\bar{u}_{\bar{h}}$, where $\bar{u}_n$ and $\bar{u}_{\bar{h}}$ are the solutions of \eqref{eq:skeleton} corresponding to $\bar{h}_n \in S_M$ and $\bar{h}$ respectively such that $\bar{h}_n \rightharpoonup \bar{h}$ in $L^2(0,T; \mathcal{H})$. However, one cannot expect to achieve the strong convergence of $\{\bar{u}_n\}$ in a straight-forward manner due to its lack of regularity and hence the lack of compactness of the embedding $\mathcal{Y}\hookrightarrow L^2(\R^d) \hookrightarrow \mathcal{Y}^\prime$. We use a-priori estimate \eqref{apriori-cond-c2} and a localized time increment estimate (cf. Lemma \ref{lem:time-increment-un-cond-c2}) for $\bar{u}_n$ and $\bar{u}_{\bar{h}}$ along with the property of the Hilbert-Schmidt operator (of the noise coefficient $\sigma$) to validate the condition \ref{C2}; see Subsection \ref{subsec:cond-c2}. 
\end{itemize}
\item[c)] 
We use Girsanov's transformation together with the formulation of the relative entropy in terms of controlled process as appeared in the Girsanov theorem and the standard $L^2$-estimation to establish the quadratic TCI inequality for the solution of \eqref{eq:spde} on $C([0,T];L^2(\R^d))$ for the bounded diffusion coefficient i.e., when $\sigma$ satisfies the inequality \eqref{cond:bound-sigma}. In particular, the laws of the strong solution of the underlying problem as a Borel probability measure concentres on on $C([0,T];L^2(\R^d))$.
 \end{itemize}
  
 \vspace{0.1cm}
The rest of the paper is organized as follows. In Section \ref{sec:technical}, we state the assumptions, basic definitions, the necessary framework and then state the main results of the paper. Section \ref{sec:existence-uni-skeleton} is devoted to prove well-posedness result for the skeleton equation \eqref{eq:skeleton}. The proofs of LDP and the quadratic TCI inequality are provided in Sections \ref{sec:proof-ldp} and \ref{sec:proof-tci} respectively. 
\section{Preliminaries and technical framework}\label{sec:technical}
In this paper, we use the letters $C$ to denote various generic constants. For any $p\in (1,\infty)$, let $p^\prime$ be the convex conjugate of $p$. Consider the space 
\begin{align*}
\mathcal{Y}:=\big\{ v \in L^2(\R^d):~~\nabla v \in L^p(\R^d)^d\big\}\,
\end{align*}
endowed with the norm 
$$ \|v\|_{\mathcal{Y}}:= \|v\|_{L^2(\R^d)} + \|\nabla v\|_{L^p(\R^d)^d},~~~v\in \mathcal{Y}\,.$$
Let $\mathcal{Y}^\prime$ be the dual of $\mathcal{Y}$. Then, $\mathcal{Y}^\prime$ has the following characterization:
$$ \mathcal{Y}^\prime=\big\{ f-{\rm div}\,g:~~ f\in L^2(\R^d),~~g\in L^{p^\prime}(\R^d)^d\big\}.$$
As mentioned in \cite{schmitz-23}, $\mathcal{Y}$ is a reflexive and separable Banach space. Moreover, the $p$-Laplace operator 
\begin{align*}
\Delta_p: &~\mathcal{Y} \goto \mathcal{Y}^\prime \\ & v\mapsto {\rm div}\big( |\nabla v|^{p-2}\nabla v\big)~~~(1<p<\infty)
\end{align*}
is well-defined. Furthermore, the triple $\mathcal{Y}\hookrightarrow L^2(\R^d) \hookrightarrow \mathcal{Y}^\prime$
form a Gelfand triple. For any two separable Hilbert spaces $\mathcal{U}$ and $\mathcal{V}$, let $\mathcal{L}_2(\mathcal{U}, \mathcal{V})$ be the space of all Hilbert-Schmidt operators from $\mathcal{U}$ to $\mathcal{V}$. Let $\mathcal{H}_0$ be a Hilbert space such that the Wiener process $W$ has $\mathbb{P}$-a.s., continuous trajectories in $C([0,T]; \mathcal{H}_0)$ and the embedding 
$\mathcal{H}\hookrightarrow \mathcal{H}_0$ is Hilbert-Schmidt.  For any separable Hilbert space $\mathcal{W}$, we denote $N_w^2(0,T;\mathcal{W})$, as the space of all square integrable $\{\mathcal{F}_t\}_{t\ge 0}$-adapted 
$\mathcal{W}$-valued process $u$ such that $\mathbb{E}\big[ \int_0^T \|u(t)\|_{\mathcal{W}}^2\,dt\big] < \infty$. 
\vspace{0.1cm}

We now recall the notion of (strong) solution for the SPDE \eqref{eq:spde} from \cite{schmitz-23}.
\begin{defi}\label{defi:strong-solution-spde}
We say that $u\in N_w^2(0,T; L^2(\R^d))$ is a strong solution of \eqref{eq:spde} with initial condition $u_0 \in L^2(\R^d)$, if 
\begin{itemize}
\item[i)] $\mathbb{P}$-a.s., $u\in C([0,T];L^2(\R^d))$, and $u\in L^q(\Omega \times (0,T); \mathcal{Y})$, where $q=\min\{ p, 2\}$.
\item[ii)] For all $t\in [0,T]$ and $\mathbb{P}$-a.s., there holds
\begin{align*}
u(t)= u_0 + \int_0^t \Delta_p(u(s))\,ds + \int_0^t \sigma(u(s))\,dW(s)\,.
\end{align*}
\end{itemize}
\end{defi}
The aim of this paper is to establish the LDP and the quadratic TCI inequality (hence measure concentration) for the
strong solution of \eqref{eq:spde}, and we will do so under the following assumptions:
\begin{Assumptions}
\item \label{A1} The initial function $u_0\in L^2(\R^d)$.
\item \label{A2}$\sigma: L^2(\R^d) \goto \mathcal{L}_2(\mathcal{H}, L^2(\R^d)) $ such that $\sigma$ is Lipschitz continuous, i.e., there exists $c_\sigma >0$ such that
\begin{align}
\|\sigma(u)-\sigma(v)\|_{\mathcal{L}_2(\mathcal{H}, L^2(\R^d))} \le c_\sigma\|u-v\|_{L^2(\R^d)},\quad \forall~u,v \in L^2(\R^d)\,.\label{inq:lip-sigma}
\end{align}
\item \label{A3}  $\sigma(\cdot)$ has linear growth i.e., there exists $\sigma_b >0$ such that
\begin{align}
\|\sigma(u)\|_{\mathcal{L}_2(\mathcal{H}, L^2(\R^d))}\le \sigma_b (1+ \|u\|_{L^2(\R^d)}),\quad \forall~u\in L^2(\R^d)\,. \label{inq: growth-sigma}
\end{align}
\end{Assumptions}
We recall the result of the well-posedness of strong solution for \eqref{eq:spde}; see \cite[Theorems 1.3 $\&$ 1.4]{schmitz-23}.
\begin{thm}\label{thm:exis-uni-spde}
Under the assumptions \ref{A1}-\ref{A3}, equation \eqref{eq:spde} has a unique strong solution in the sense of Definition \ref{defi:strong-solution-spde}
\end{thm}
\subsection{Large deviation principle}  Recall some basic definitions and results of large deviation theory. Let $\{X_\epsilon\}_{\epsilon > 0}$  be a sequence of random variable defined on a given probability space $(\Omega, \mathcal{F}, \mathbb{P})$ taking values in a Polish space $\mathcal{Z}$. It is well-known that, the large deviation principle characterizes the exponential decay of the remote tails of some sequences of probability distributions. The rate of such decay is described by the ``rate function".
     \begin{defi}[Rate Function]
      A function  $ {\tt I}:\mathcal{Z} \rightarrow [0,\infty]$ is called a {\it rate function} if ${\tt I}$ is lower semi-continuous. A rate function ${\tt I}$ is called a good rate function if the level set $\{x\in \mathcal{Z} : {\tt I}(x) \leq M \}$ is compact for each $M < \infty$.
\end{defi}
\begin{defi} [Large deviation principle]
The sequence  $\{X_\epsilon\}_{\epsilon > 0}$  is said to satisfy the large deviation principle on $\mathcal{Z}$ with rate function ${\tt I}$ if the following conditions hold:
\begin{itemize}
\item [i)] For any open set $G\subset \mathcal{Z}$, 
\begin{align*}
  -\underset{x\in G}\inf\, {\tt I}(x) \leq \underset{\epsilon \longrightarrow 0}{\liminf}~ \epsilon^{2}\log \big(\mathbb{P}(X_\epsilon \in G)\big)
\end{align*}
\item[ii)] For any closed set $K \subset \mathcal{Z}$, 
\begin{align*}
       \underset{\epsilon \longrightarrow 0}{\limsup} ~\epsilon^{2} \log \big(\mathbb{P}(X_\epsilon \in K)\big)
        \leq  -\underset{x\in K} \inf\, {\tt I}(x)\,.
\end{align*}
\end{itemize}
\end{defi}
Let us recall the definition of the Laplace principle cf.~\cite{dembo,dupuis}.

\begin{defi}
The sequence  $\{X_\epsilon\}_{\epsilon > 0}$  is said to satisfy the Laplace principle on $\mathcal{Z}$ with rate function ${\tt I}$ if for each $h\in C_b(\mathcal{Z}; \R)$, the following holds:
\begin{align*}
\lim_{\epsilon \rightarrow 0} \eps^2~ \log \Big\{\mathbb{E}\Big[ \exp\Big(-\frac{h(X_\eps)}{\eps^2} \Big)\Big]\Big\}
=- \inf_{x\in \mathcal{Z}} \big\{ h(x) + {\tt I}(x)\big\}\,.
\end{align*}
\end{defi}
Thanks to  \cite[Varadhan's lemma]{varadhan-84} and \cite[Bryc's theorem]{dembo}, we have the following theorem regarding the equivalence between the LDP and the Laplace principle.
\begin{thm}
The sequence  $\{X_\epsilon\}_{\epsilon > 0}$  on $\mathcal{Z}$ satisfies the LDP with a good rate function ${\tt I}$ if and only if it satisfies the Laplace principle with same rate function ${\tt I}$. 
\end{thm}
For $\eps>0$, consider the evolutionary PDE driven by small multiplicative noise:
\begin{equation}\label{eq:ldp}
\begin{aligned}
&du_\eps(t)- {\rm div}_x(|\nabla u_\eps|^{p-2}\nabla u_\eps(t))\,dt= \eps\, \sigma(u_\eps(t))\,dW(t),~(t,x)\in (0,T]\times \R^d, \\
&u_\eps(0,x)=u_0(x),~~x\in \R^d\,.
\end{aligned}
\end{equation}
Under the assumptions \ref{A1}-\ref{A3}, equation \eqref{eq:ldp} has a unique strong solution $u_\eps \in L^2(\Omega; C([0,T];L^2(\R^d)))\cap L^q(\Omega\times(0,T); \mathcal{Y})$, cf.~Theorem \ref{thm:exis-uni-spde}.  Furthermore, from the infinite dimensional version of Yamada-Watanabe theorem in \cite{Rock-2008}, there exists a Borel-measurable function $\mathcal{G_{\epsilon}} : C([0,T];\mathcal{H}_0) \rightarrow C([0,T];L^2(\R^d))\cap L^q(0,T;\mathcal{Y})\subset \mathcal{Z}: = C([0,T];L^{2}(\R^d))$ such that $u_{\epsilon} := \mathcal{G}_{\epsilon}(W)$ a.s.  
\vspace{0.1cm}

To prove LDP, we need to introduce the deterministic skeleton equation. We recall that $\left\{W(t) \right\}_{t \geq 0}$ is a cylindrical Wiener process on $\mathcal{H}$ with respect to the filtered probability space $\big(\Omega, \mathbb{P}, \mathcal{F},  \{ \mathcal{F}_t\}_{t \geq 0} \big)$ such that its path is continuous on $\mathcal{H}_0$. Consider the following sets:
\begin{align}
      &\mathcal{A} := \Big\{ \phi :  \text{ $\phi$  is a $\mathcal{H}$-valued  $\{\mathcal{F}_t\}$-adapted process  such that }  \int_0^T\|\phi(s)\|_{\mathcal{H}}^2 \, {\rm d}s<  \infty,  \, \mathbb{P}-a.s.\Big\}, \notag \\
      &S_M := \Big\{h \in L^2([0,T]; \mathcal{H}) : \int_0^T\|h(s)\|_{\mathcal{H}}^2\,{\rm d}s\leq M\Big\}, ~~~
    \mathcal{A}_M :=\Big\{\phi \in \mathcal{A} : \phi(\omega) \in S_M,  \mathbb{P}-a.s.\Big\}. \notag
\end{align}
It is well known that $S_M$ endowed with the weak topology is a Polish space. We always consider the weak topology on $S_M$ for our later analysis. 
For $h \in L^2(0,T; \mathcal{H})$, we consider the following skeleton equation~(deterministic)
\begin{equation}
\label{eq:skeleton}
   \begin{aligned}
            du_{h}  - {\rm div}_x ( |\nabla u_{h}|^{p-2}\nabla u_{h})\,dt &=  \sigma(u_{h})\,h(t)\,{\rm d}t,~~~(t,x)\in (0,T] \times \R^d, \\
            u_{h}(0,x) &= u_0(x),~~~x\in \R^d\,.
        \end{aligned}
\end{equation}
\begin{thm}\label{thm:skeleton}
Let the assumptions \ref{A1}-\ref{A3} hold true. Then there exists a unique solution $u_h \in C([0,T];L^2(\R^d))\cap L^{q}(0,T; \mathcal{Y})$ of the skeleton equation \eqref{eq:skeleton}.  Moreover, for fixed $N\in \mathbb{N}$, there exists a constant $C_N>0$ such that 
\begin{align}
\sup_{h\in S_N} \Big\{ \sup_{t\in [0,T]} \|u_h(t)\|_{L^2(\R^d)}^2 + \int_0^T \|u_h(t)\|_{\mathcal{Y}}^q\,dt \Big\} \le C_N\,. \label{esti:uniform-skeleton-original}
\end{align}
\end{thm}
This allow us to define $\mathcal{G}_{0} : C([0,T];\mathcal{H}_0) \rightarrow  \mathcal{Z}$ by
\begin{equation}
    \mathcal{G}_{0}(\phi) := \begin{cases}
\displaystyle   u_{h} \quad  \text{if $\phi = \int_{0}^{\cdot} h(s) \,{\rm d}s$  for some $ h \in L^2(0,T;\mathcal{H})$}, \notag \\
     0 \quad \text{otherwise}.
    \end{cases}
\end{equation}
Our main theorem on LDP is stated as follows.
\begin{thm}\label{thm:ldp}
Let the assumptions \ref{A1} - \ref{A3} hold true and $1<p<\infty$ be given. For any $\epsilon > 0$, let  $u_{\epsilon}$ be the unique strong solution of \eqref{eq:ldp}. Then $\{u_{\epsilon} \}_{\epsilon > 0}$ satisfies the LDP on $\mathcal{Z} = C([0,T];L^2(\R^d))$ with the good rate function ${\tt I}$ given by 
\begin{align} 
  {\tt I}(f)=
\displaystyle  \underset{\big\{h \in L^2(0,T; \mathcal{H}) : f =  \mathcal{G}_{0} \big(\int_0^{\cdot}h(s) \, {\rm d}s \big)\big\}}\inf  \bigg \{ \frac{1}{2} \int_0^T\|h(s)\|_{\mathcal{H}}^2 \, {\rm d}s\bigg \}\,, \notag
 \end{align}
 where infimum over an empty set is taken as $+ \infty$.
\end{thm}

\vspace{0.1cm}

\subsection{Transportation cost information inequality}
Let $\mathcal{P}({\tt X})$ be the set of all probability measures on $({\tt X},{\tt d})$, where  $({\tt X},{\tt d})$ is a given metric space equipped with the Borel $\sigma$-field $\mathcal{B}$.  For any $\mu, \nu \in \mathcal{P}({\tt X})$ and $r \in [1,\infty)$,  we define the $L^r$-Wasserstein distance between  $\mu$ and $\nu$, denoted by ${\tt W}_r(\nu, \mu) $,  as  (cf.~ \cite{villan} )
 \begin{equation} 
 {\tt W}_r(\nu, \mu) := \bigg [\inf \int_{{\tt X} \times {\tt X}} {\tt d}(x,y)^r\pi({\rm d}x,{\rm d}y) \bigg ]^{\frac{1}{r}} \\
 = \inf \bigg \{\big[\mathbb{E}({\tt d}(\mathcal{X},\mathcal{Y})^r)\big]^{\frac{1}{r}},~~~ law(\mathcal{X}) = \mu, ~~ law(\mathcal{Y}) = \nu \bigg \}, \notag
  \end{equation}
 where the infimum is taken over all the joint probability measures $\pi$ on ${\tt X} \times {\tt X}$ with marginals $\mu$ and $\nu$.
We say that $\nu$ is absolutely continuous with respect to $\mu$, denoted by $\nu  \ll  \mu$, if $\nu(A)=0$ for any  $A\in \mathcal{B}$ such that $\mu(A)=0$. The relative entropy (Kullback information) of $\nu$ with respect to $\mu$ is defined by 
 \begin{equation*}
     \mathcal{H}(\nu | \mu) = \begin{cases}
\displaystyle      \int_{{\tt X}} \log(\frac{{\rm d} \nu}{{\rm d} \mu}) {\rm d}\nu \quad  if \, \nu  \ll  \mu, \\ 
       +\infty \quad \quad \text{otherwise}. 
     \end{cases}
 \end{equation*}
 
\begin{defi}[Transportation cost information inequality] \label{defi:tci}
A measure $\mu$ satisfies the $L^r$-transportation cost information inequality if there exists a constant ${\tt C}>0$ such that for all probability measure $\nu$,
\begin{equation} \label{eq:reltci}
{\tt W}_r(\nu, \mu) \leq \sqrt{2{\tt C}\mathcal{H}(\nu|\mu)}.
\end{equation}
\end{defi} 
 We denote the relation in \eqref{eq:reltci} as $\mu \in T^{r}({\tt C})$. The case $r = 2, T^{2}({\tt C})$ is referred to as the quadratic TCI inequality. Thanks to H\"{o}lder's inequality, one can easily see that 
$$T^{q}({\tt C})\subseteq T^{r}({\tt C})~\text{ for $1\le r\le q$}$$
and hence the property $ T^{1}({\tt C})$ is weaker than $ T^{2}({\tt C})$.  We now ready to state the main result regarding the quadratic TCI inequality. 
\begin{thm} \label{thm:maintci} 
Let the assumptions \ref{A1}-\ref{A3} hold true. Moreover, in addition, $\sigma(\cdot)$ satisfies the following boundedness property: there exists a constant $\bar{\sigma}_b >0$ such that 
\begin{align}
\|\sigma(u)\|_{\mathcal{L}_2(\mathcal{H}, L^2(\R^d))}\le \bar{\sigma}_b, \quad \forall~u \in L^2(\R^d)\,. \label{cond:bound-sigma}
\end{align}
 Then the law $\mu$ of the solution of \eqref{eq:spde} satisfies the quadratic TCI inequality on ${\tt X}:= C([0,T];L^2(\R^d))$ i.e., $\mu \in T^{2}({\tt C})$, where ${\tt C}:=2\bar{\sigma}_b^2 \exp\big\{2T\big(1 + 33\,c_\sigma^2\big)\big\} $. In particular, $\mu$  concentrates as a Borel measure on $C([0,T];L^2(\R^d))$.
\end{thm}

\section{Wellposedness of Skeleton equation}\label{sec:existence-uni-skeleton}
This section is devoted to prove existence and uniqueness of the solution of  skeleton equation \eqref{eq:skeleton}  i.e.,  Theorem \ref{thm:skeleton}.
\subsection{Uniqueness of solutions for \eqref{eq:skeleton}} \label{subsec:uni-skeleton}
For a fixed $h\in L^2(0,T;\mathcal{H})$, let $u_1$ and $u_2$ be two solutions of the skeleton equation \eqref{eq:skeleton}. Then using the chain-rule formula, the fundamental inequality
\begin{align}
\langle |\eta|^{p-2} \eta - |\zeta|^{p-2}\zeta, \eta-\zeta \rangle \ge 0,\quad \forall~ \eta, \zeta \in \R^d;~~~1<p< \infty\,, \label{inq:monotone}
\end{align}
the Cauchy-Schwartz inequality, Young's inequality and the assumptions \ref{A2}-\ref{A3}, we have, for $u(t):=u_1(t)-u_2(t)$,
\begin{align}
\|u(t)\|_{L^2(\R^d)}^2 & \le C \int_0^t \|u(s)\|_{L^2(\R^d)}^2\,ds + C \int_0^t \|\sigma(u_1)-\sigma(u_2)\|_{\mathcal{L}_2(\mathcal{H}, L^2(\R^d))}^2 \|h(s)\|_{\mathcal{H}}^2\,ds \notag \\
& \le C \int_0^t \big( 1 + \|h(s)\|_{\mathcal{H}}^2\big)\|u(s)\|_{L^2(\R^d)}^2\,ds\,. \notag
\end{align}
Hence, in view of Gronwall's inequality and the fact that $h\in L^2(0,T; \mathcal{H})$, we see that $u_1(t)=u_2(t)$ for all $t\in [0,T]$ and a.e. in $\R^d$. Hence the solution to \eqref{eq:skeleton} is unique. 

\subsection{Well-posedness of auxiliary skeleton equation }\label{subsec:existence-skeleton-bounded-control}
For any $h\in L^\infty(0,T; \mathcal{H})$, consider the following  auxiliary skeleton equation
\begin{equation}\label{eq:skeleton-auxi}
\begin{aligned}
&du-{\rm div}_x(|\nabla u|^{p-2}\nabla u) dt =  \sigma(\rho) h(t)\,dt; \quad 
u(0,x)=u_0(x),~~x\in \R^d\,,
\end{aligned}
\end{equation}
where $\rho \in L^2(0,T; L^2(\R^d))$ is fixed. 
We prove that \eqref{eq:skeleton-auxi} has a unique solution. As mentioned earlier, one cannot apply generalized monotonicity method 
\cite{ren-2007}.  Taking primary motivation from \cite{schmitz-23}, we use  a semi-implicit Euler Maruyama scheme. To do so, we introduce the semi-implicit time discretization as follows. For $N \in \mathbb{N}$, let $0 = t_0 < t_1 < ... < t_N = T$ be a uniform partition of $[0,T]$ with $\tau := \frac{T}{N} = t_{k+1}-t_k$ for all $k=0,1,...,N-1$. Consider the semi-implicit Euler Maruyama scheme of \eqref{eq:skeleton-auxi}:  for $k=0,1,...,N-1$,
 \begin{equation}\label{eq:discrete}
u_{k+1} - u_{k} - \tau {\rm div}( |\nabla u_{k+1}|^{p-2}\nabla u_{k+1}\big) =\int_{t_k}^{t_{k+1}} {\tt A}(s)\,ds
\end{equation}
where ${\tt A}(s):=  \sigma(\rho(s))h(s)$. 
Note that, 
 the operator 
 \begin{align*}
 \mathcal{A}_\tau:& \mathcal{Y} \goto \mathcal{Y}^\prime \\
 &\nu \mapsto \nu -\tau \Delta_p(\nu)
 \end{align*}
 is hemicontinuous, strict monotone, coercieve and bounded; see \cite{schmitz-23}. Moreover, $\mathcal{A}_\tau^{-1}: \mathcal{Y}^\prime \goto \mathcal{Y}$ is demi-continuous. Since for each $k$, $F_k:= u_k + \int_{t_k}^{t_{k+1}}{\tt A}(s)\,ds \in \mathcal{Y}^\prime$, by Minty-Browder theorem, there exists a unique element $u_{k+1}\in \mathcal{Y}$ such that $\mathcal{A}_\tau (u_{k+1})=F_k$. In other words, there exists unique $u_{k+1}\in \mathcal{Y}$ such that \eqref{eq:discrete} holds. 
\begin{defi}
  For $N \in \mathbb{N}$, $\tau > 0$, and $t\in [0,T]$, we define the right-continuous step function
          \[ u_{\tau}^r(t) = \sum_{k=0}^{N-1} u_{k+1} \chi_{[t_k,t_{k+1})} (t),~~u_\tau^r(T)=u_N, \]
            the left-continuous step function
          \[u_{\tau}^l(t) = \sum_{k=0}^{N-1} u_{k} \chi_{(t_k,t_{k+1}]} (t),~~u_\tau^l(0)=u_0, \]
        and the piecewise affine functions
         \begin{align*}
          & u_{\tau}(t) = \sum_{k=0}^{N-1} \left (\frac{u_{k+1}-u_{k}}{\tau} (t-t_k) + u_{k}  \right ) \chi_{[t_k,t_{k+1})} (t),~~u_{\tau}(T) =u_N, \\
           & \Phi_\tau(t)= \sum_{k=0}^{N-1} \left (\frac{\Phi_{k+1}-\Phi_{k}}{\tau} (t-t_k) + \Phi_{k}  \right ) \chi_{[t_k,t_{k+1})} (t),~~\Phi_{\tau}(T) =\Phi_N,
 \end{align*}
 where $\Phi_k=\int_0^{t_k} {\tt A}(s)\,ds$. 
\end{defi}
\subsubsection{\bf A-priori estimate for $\{u_\tau\}$}
 Taking test function $v=u_{k+1}$ in \eqref{eq:discrete}, and using the Cauchy-Schwartz inequality, Young's inequality, and the identity
 \begin{align}
(x-y)x=\frac{1}{2}\big\{ x^2 -y^2 + (x-y)^2\big\},~~\forall~x,y \in \R\,, \label{identity-1}
 \end{align}
 we get
 \begin{align}
& \frac{1}{2} \Big\{ \|u_{k+1}\|_{L^2(\R^d)}^2 - \|u_k\|_{L^2(\R^d)}^2 + \|u_{k+1}-u_k\|_{L^2(\R^d)}^2\Big\} + \tau \|\nabla u_{k+1}\|_{L^p(\R^d)^d}^p \notag \\
& \le \frac{1}{4}\|u_{k+1}-u_k\|_{L^2(\R^d)}^2 + \frac{\tau}{4} \|u_k\|_{L^2(\R^d)}^2 + 2 \int_{t_k}^{t_{k+1}} \|{\tt A}(s)\|_{L^2(\R^d)}^2\,ds\,. \label{esti-1-auxi}
 \end{align}
Discarding the non-negative terms, and then taking sum over $k=0,1,\ldots, n-1$ for fixed $n\in \{1,2,\ldots, N\}$, we obtain
\begin{align}
\|u_n\|_{L^2(\R^d)}^2 \le \|u_0\|_{L^2(\R^d)}^2 + \frac{\tau}{2} \sum_{k=0}^{n-1} \|u_k\|_{L^2(\R^d)}^2 + 4 \int_0^T \|{\tt A}(s)\|_{L^2(\R^d)}^2\,ds\,.\notag
\end{align}
Using the discrete Gronwall's lemma and the fact that $\int_0^T \|{\tt A}(s)\|_{L^2(\R^d)}^2\,ds < + \infty$, we conclude that
\begin{align}
\sup_{0\le n\le N} \|u_n\|_{L^2(\R^d)}^2 \le C \label{esti:uni-discrete-l2}
\end{align}
for some constant $C>0$, independent of $\tau$. Furthermore, by using \eqref{esti:uni-discrete-l2} in \eqref{esti-1-auxi}, we see that
\begin{equation}\label{esti:uni-discrete-diff-lp}
\begin{aligned}
\tau \sum_{k=0}^{N-1} \|\nabla u_{k+1}\|_{L^p(\R^d)^d}^p & \le \frac{1}{2} \|u_0\|_{L^2(\R^d)}^2 + \frac{\tau}{4} \sum_{k=0}^{N-1}\|u_k\|_{L^2(\R^d)}^2 + 2 \int_0^T \|{\tt A}(s)\|_{L^2(\R^d)}^2\,ds \le C\,, \\
\sum_{k=0}^{N-1}\|u_{k+1}-u_k\|_{L^2(\R^d)}^2 &\le 2 \|u_0\|_{L^2(\R^d)}^2 + \tau \sum_{k=0}^{N-1}\|u_k\|_{L^2(\R^d)}^2 + 8 \int_0^T  \|{\tt A}(s)\|_{L^2(\R^d)}^2\,ds \le C\,.
\end{aligned}
\end{equation}
In view of \eqref{esti:uni-discrete-l2}, \eqref{esti:uni-discrete-diff-lp}, and the definition of $u_\tau^r, u_\tau^l$ and $u_\tau$, one may easily see that
\begin{equation}\label{esti:apriori-cont-aux}
\begin{aligned}
& \sup_{t\in [0,T]}\|u_\tau^r(t)\|_{L^2(\R^d)}^2=\sup_{t\in [0,T]}\|u_\tau^l(t)\|_{L^2(\R^d)}^2=\sup_{t\in [0,T]}\|u_\tau(t)\|_{L^2(\R^d)}^2 \le C\,, \\
& \int_0^T \|\nabla u_\tau^r(t)\|_{L^p(\R^d)^d}^p\,dt \le  \tau \sum_{k=0}^{N-1} \|\nabla u_{k+1}\|_{L^p(\R^d)^d}^p \le C\,, \\
& \int_0^T \|u_\tau ^r(t)- u_{\tau}^l(t)\|_{L^2(\R^d)}^2\,dt = \sum_{k=0}^{N-1} \int_{t_k}^{t_{k+1}} \|u_{k+1}-u_k\|_{L^2(\R^d)}^2\,dt \le C \tau\,, \\
& \int_0^T \|u_\tau ^r(t)- u_{\tau}^l(t)\|_{L^2(\R^d)}^2\,dt = \sum_{k=0}^{N-1} \int_{t_k}^{t_{k+1}} \|\frac{u_{k+1}-u_k}{\tau}(t-t_k)\|_{L^2(\R^d)}^2\,dt \le C \tau\,.
\end{aligned}
\end{equation}
\subsubsection{\bf Convergence analysis }
 Thanks to \eqref{esti:apriori-cont-aux}, there exists $u_* \in L^2(0,T; L^2(\R^d))\cap L^q(0,T; \mathcal{Y})$ with $q=\min\{2,p\}$ such that as $\tau\goto 0$
 \begin{align}
u_\tau^r \rightharpoonup u_*,~~u_\tau^l \rightharpoonup u_*,~~u_\tau \rightharpoonup u_* ~~\text{in}~~L^2(0,T; L^2(\R^d))\,. \label{conv-weak-l2}
 \end{align}
 Observe that, for any $t\in [t_k, t_{k+1})$ and $k\in \{0,1, \ldots, N-1\}$
 \begin{align*}
& \Big\| \Phi_\tau(t)- \int_0^t {\tt A}(s)\,ds \Big\|_{L^2(\R^d)}^2 \\
& \le 2 \Big( \Big\| \int_{t_k}^t {\tt A}(s)\,ds\Big\|_{L^2(\R^d)}^2 + \Big\| \int_{t_k}^{t_{k+1}} {\tt A}(s)\,ds\Big\|_{L^2(\R^d)}^2\Big) \le 2 \tau \int_{t_k}^{t_{k+1}}\|{\tt A}(s)\|_{L^2(\R^d)}^2\,ds\,.
 \end{align*}
 Thus, one has
 \begin{align*}
\int_0^T \Big\| \Phi_\tau(t)- \int_0^t {\tt A}(s)\,ds \Big\|_{L^2(\R^d)}^2\,dt & \le \sum_{k=0}^{N-1} \int_{t_k}^{t_{k+1}} \Big\| \Phi_\tau(t)- \int_0^t {\tt A}(s)\,ds \Big\|_{L^2(\R^d)}^2\,dt \\
& \le 2 \tau^2 \int_0^T \|{\tt A}(s)\|_{L^2(\R^d)}^2\,ds \le C \tau\,.
 \end{align*}
 In other words,
 \begin{align}
\Phi_\tau \goto \int_0^{\cdot} {\tt A}(s)\,ds ~~~\text{in}~~L^2(0,T; L^2(\R^d))~~\text{as}~~\tau\goto 0\,. \label{conv:strong-l2-aux}
 \end{align}
 Thanks to H\"{o}lder's inequality, we have, for any $t\in [t_k, t_{k+1})$ and $k\in \{ 0,1, \ldots, N-1\},$
 \begin{align*}
\Big\|\frac{\partial}{\partial t}\big( u_\tau(t)- \Phi_\tau(t)\big)\Big\|_{\mathcal{Y}^\prime}
\le \sup_{\psi \in \mathcal{Y},~\|\psi\|_{\mathcal{Y}}\le 1} \int_{\R^d}|\nabla u_{k+1}|^{p-1}|\nabla \psi(x)|\,dx \le \|\nabla u_{k+1}\|_{L^p(\R^d)^d}^{p-1}\,.
 \end{align*}
 Thus, using \eqref{esti:apriori-cont-aux} and the above calculation, we see that
 \begin{align}
\int_0^T \Big\|\frac{\partial}{\partial t}\big( u_\tau(t)- \Phi_\tau(t)\big)\Big\|_{\mathcal{Y}^\prime}^{p^\prime}\,dt \le \sum_{k=0}^{N-1} \int_{t_k}^{t_{k+1}} \|\nabla u_{k+1}\|_{L^p(\R^d)^d}^{p}\le C\,. \label{esti:boundedness-time-deri}
 \end{align}
Hence, in view of \eqref{esti:boundedness-time-deri}, there exists a $Y\in L^{p^\prime}(0,T;\mathcal{Y}^\prime)$ such that
$\frac{\partial}{\partial t}\big( u_\tau(t)- \Phi_\tau(t)\big) \rightharpoonup Y$ in $L^{p^\prime}(0,T;\mathcal{Y}^\prime)$. Thanks to integration by parts formula and the convergence results in \eqref{conv-weak-l2} and \eqref{conv:strong-l2-aux}, we conclude that
\begin{align}
\frac{\partial}{\partial t}\big( u_\tau(t)- \Phi_\tau(t)\big) \rightharpoonup 
\frac{\partial}{\partial t}\big( u_\tau(t)- \int_0^t {\tt A}(s)\,ds\big)~~\text{in}~~L^{p^\prime}(0,T;\mathcal{Y}^\prime)~~\text{as}~~\tau \goto 0\,. \label{conv:weak-time-deri}
\end{align}
Since $\{\nabla u_\tau^r\}$ is bounded in $L^p(0,T; L^p(\R^d)^d)$, it is easy to see that the sequence 
$\{ |\nabla u_\tau^r|^{p-2}\nabla u_\tau^r\}$ is bounded in $L^{p^\prime}(0,T;L^{p^\prime}(\R^d)^d)$ and hence there exists a not relabeled subsequence such that
\begin{align}
|\nabla u_\tau^r|^{p-2}\nabla u_\tau^r \rightharpoonup  G ~~\text{in}~~L^{p^\prime}(0,T;L^{p^\prime}(\R^d)^d)~~\text{as}~~\tau \goto 0 \label{conv:weak-lp}
\end{align}
for some ${\tt G}\in L^{p^\prime}(0,T;L^{p^\prime}(\R^d)^d)$. 
\subsubsection{\bf Well-posedness of solution of auxiliary skeleton equation:}  Following the proof of \cite[Proposition $2.11$]{schmitz-23} along with \eqref{conv:weak-time-deri}, \eqref{conv:weak-lp} and the fact that 
\begin{align*}
\frac{\partial}{\partial t}\big( u_\tau(t)- \Phi_\tau(t)\big)- {\rm div}_x \big( |\nabla u_\tau^r|^{p-2}\nabla u_\tau^r\big)=0~~\text{in}~~L^{p^\prime}(0,T; \mathcal{Y}^\prime),
\end{align*}
we infer that $u_* \in C[(0,T]; L^2(\R^d))$ and satisfies the following equation: for all $t\in [0,T]$
\begin{align}
u_*(t)-u_0  - \int_0^t {\rm div}_x {\tt G}\,ds= \int_0^t {\tt A}(s)\,ds\,. \label{eq:limit-skeleton-aux}
\end{align}
We now identify the limit ${\tt G}$. We take a test function $u_{k+1}$ in \eqref{eq:discrete}, use the identity \eqref{identity-1} and sum over $k=0,1,\ldots, N-1$ along with the definition of $u_\tau$ and $u_\tau^r$, we have
\begin{align}
\frac{1}{2}\|u_\tau(T)\|_{L^2(\R^d)}^2 + \int_0^T \int_{\R^d} |\nabla u_\tau^r|^{p-2} \nabla u_\tau^r \cdot \nabla u_\tau^r \,dx\,dt - \int_0^T \int_{\R^d} {\tt A}(s) u_\tau^r(s)\,dx\,ds \le \frac{1}{2} \|u_0\|_{L^2(\R^d)}^2\,. \label{inq:weak-limit-identification-1}
\end{align}
On the other hand, from \eqref{eq:limit-skeleton-aux}, we have
\begin{align}
\frac{1}{2}\|u_*(T)\|_{L^2(\R^d)}^2 + \int_0^T \int_{\R^d} {\tt G} \cdot \nabla u_*(s) \,dx\,dt - \int_0^T \int_{\R^d} {\tt A}(s) u_*(s)\,dx\,ds = \frac{1}{2} \|u_0\|_{L^2(\R^d)}^2\,. \label{inq:weak-limit-identification-2}
\end{align}
Subtracting \eqref{inq:weak-limit-identification-2} from \eqref{inq:weak-limit-identification-1}, we get 
\begin{align}
& \frac{1}{2}\big\{ \|u_\tau(T)\|_{L^2(\R^d)}^2-\|u_*(T)\|_{L^2(\R^d)}^2 \big\} + \Big\{ 
\int_0^T \int_{\R^d} |\nabla u_\tau^r|^{p-2} \nabla u_\tau^r \cdot \nabla u_\tau^r \,dx\,dt -
\int_0^T \int_{\R^d} {\tt G} \cdot \nabla u_*(s) \,dx\,dt \Big\} \notag \\
& \le 
\Big\{ \int_0^T \int_{\R^d} {\tt A}(s) u_*(s)\,dx\,ds - \int_0^T \int_{\R^d} {\tt A}(s) u_*(s)\,dx\,ds\Big\}\,. \label{inq:weak-limit-identification-3}
\end{align}
Following a similar argument as invoked in the proof of \cite[Lemma 2.12]{schmitz-23}, we infer that $u_\tau(T)\rightharpoonup u_*(T)$ in $L^2(\R^d)$ and hence by the weak lower semicontinuity of the norm, we get
\begin{align}
\liminf_{\tau \goto 0} \big\{ \|u_\tau(T)\|_{L^2(\R^d)}^2-\|u_*(T)\|_{L^2(\R^d)}^2 \big\} \ge 0\,. \label{conv:weak-limit-identification-1}
\end{align}
Again, thanks to \eqref{conv:strong-l2-aux}, it is easy to see that
\begin{align}
\lim_{\tau\goto 0}\int_0^T \int_{\R^d} {\tt A}(s) u_*(s)\,dx\,ds =\int_0^T \int_{\R^d} {\tt A}(s) u_*(s)\,dx\,ds\,. \label{conv:weak-limit-identification-2}
\end{align}
Using \eqref{conv:weak-limit-identification-1} and \eqref{conv:weak-limit-identification-2} in \eqref{inq:weak-limit-identification-3}, we obtain
\begin{align}
\limsup_{\tau\goto 0} \int_0^T \int_{\R^d} |\nabla u_\tau^r|^{p-2} \nabla u_\tau^r \cdot \nabla u_\tau^r \,dx\,dt \le \int_0^T \int_{\R^d} {\tt G} \cdot \nabla u_*(s) \,dx\,dt\,.
\label{inq:weak-limit-identification-4}
\end{align}
Let ${\tt v}\in L^p(0,T; L^p(\R^d)^d)$. Thanks to \eqref{inq:monotone} and \eqref{inq:weak-limit-identification-4}, and the fact that $\nabla u_\tau^r \rightharpoonup \nabla u_*$ in $L^p(0,T; L^p(\R^d)^d)$, one has
\begin{align*}
0 &\le \limsup_{\tau\goto 0} \int_{0}^T\int_{\R^d} \big(|\nabla u_\tau^r|^{p-2}\nabla u_\tau^r -
|{\tt v}|^{p-2}{\tt v}\big) \cdot \big( \nabla u_\tau^r - {\tt v}\big)\,dx\,dt \\
& \le \int_0^T\int_{\R^d} {\tt G}\cdot \nabla u_*\,dx\,ds - \int_0^T \int_{\R^d} |{\tt v}|^{p-2}{\tt v} \cdot \nabla u_*\,dx\,ds - \int_0^T \int_{\R^d} \big( {\tt G}-|{\tt v}|^{p-2}{\tt v} \big) \cdot {\tt v}\,dx\,ds \\
&= \int_0^T\int_{\R^d} \big( {\tt G}-|{\tt v}|^{p-2}{\tt v} \big) \cdot(\nabla u_*- {\tt v})\,dx\,ds\,.
\end{align*}
Taking ${\tt v}= \nabla u_* + \lambda \psi $ with $\lambda \in \R$ and $\psi \in L^{p}(0,T; L^p(\R^d)^d)$ in the above expression, one can easily get that
$$ {\tt G}= |\nabla u_*|^{p-2}\nabla u_* ~~~\text{in}~~L^{p^\prime}(0,T; L^{p^\prime}(\R^d)^d)\,.$$
In other words, we conclude that $u_*$ is a solution of \eqref{eq:skeleton-auxi} for $h\in L^\infty(0,T;\mathcal{H})$. Moreover, similar argument as invoked in Subsection \ref{subsec:uni-skeleton} reveals that solution of \eqref{eq:skeleton-auxi} is unique. 
\subsection{Well-posedness result for equation \eqref{eq:skeleton} in case of $h\in L^\infty(0,T; \mathcal{H})$ }
 In this subsection, we show that equation \eqref{eq:skeleton} has a unique solution in the case $h\in L^\infty(0,T; \mathcal{H})$. In view of Subsection \ref{subsec:existence-skeleton-bounded-control}, the solution operator
\begin{align*}
\mathcal{S}:& L^2(0,T; L^2(\R^d))\goto L^2(0,T; L^2(\R^d)) \\
& \rho \mapsto \mathcal{S}(\rho)=u_{\rho}\,,
\end{align*}
where $u_\rho \in L^2(0,T; L^2(\R^d))\cap L^q(0,T; \mathcal{Y})$ is the unique solution of \eqref{eq:skeleton-auxi}, is well-defined. For any $\rho_1, \rho_2\in L^2(0,T; L^2(\R^d))$, let $u:= \mathcal{S}(\rho_1)-\mathcal{S}(\rho_2)$. Then using the chain-rule, the fundamental inequality \eqref{identity-1}, the Cauchy-Schwartz inequality, Young's inequality and \eqref{inq:lip-sigma}, we get
\begin{align*}
\frac{1}{2}\|u(t)\|_{L^2(\R^d)}^2 & \le \int_0^t \|\sigma(\rho_1)-\sigma(\rho_2)\|_{\mathcal{L}_2(\mathcal{H}, L^2(\R^d))}^2\|h(t)\|_{\mathcal{H}}^2\,ds \\
& \le  c_\sigma^2 \|h\|_{L^\infty(0,T;\mathcal{H})}^2\, \int_0^t \|\rho_1-\rho_2\|_{L^2(\R^d)}^2\,ds\,. 
\end{align*}
Thus, an application of integration by parts formula reveals that for any ${\tt M}:=2  c_\sigma^2 \|h\|_{L^\infty(0,T; \mathcal{H})}^2 < \alpha$,
\begin{align}
& \int_0^T e^{-\alpha t}\| (\mathcal{S}(\rho_1)-\mathcal{S}(\rho_2))(t)\|_{L^2(\R^d)}^2\,dt \notag \\
& \le {\tt M} \int_0^T e^{-\alpha t} \Big( \int_0^t \|\rho_1-\rho_2\|_{L^2(\R^d)}^2\,ds\Big)\,dt \notag 
\\
&= \frac{{\tt M}}{\alpha} \int_0^T e^{-\alpha t} \|\rho_1(t)-\rho_2(t)\|_{L^2(\R^d)}^2\,dt -  \frac{{\tt M}}{\alpha} e^{-\alpha T} \int_0^T  \|\rho_1(t)-\rho_2(t)\|_{L^2(\R^d)}^2\,dt \notag  \\
& < \int_0^T e^{-\alpha t} \|\rho_1(t)-\rho_2(t)\|_{L^2(\R^d)}^2\,dt\,. \label{inq:contraction}
\end{align}
Consider the space
$$ \mathbb{K}:=\big\{ u: [0,T]\goto L^2(\R^d):~~\int_0^T \|u(t)\|_{L^2(\R^d)}^2\, e^{-\alpha t}\,dt < + \infty\big\}.$$
 Then the space $\mathbb{K}$ endowed with the norm 
$$\|u\|_{\mathbb{K}}^2:=\int_0^T \|u(t)\|_{L^2(\R^d)}^2 e^{-\alpha t}\,dt $$
is a Banach space and the norms 
$\|\cdot\|_{\mathbb{K}}$ and $\|\cdot\|_{L^2(0,T; L^2(\R^d))}$ are equivalent. In view of \eqref{inq:contraction}, we apply the Banach fixed point theorem to have a unique element $\bar{u} \in \mathbb{K}$ and hence in $L^2(0,T; L^2(\R^d))$ such that $\mathcal{S}(\bar{u})=\bar{u}$. In other words, equation \eqref{eq:skeleton} has a unique solution in the case $h\in L^\infty(0,T; \mathcal{H})$. 
\subsection{Existence analysis for \eqref{eq:skeleton} for general $h\in L^2(0,T; \mathcal{H})$:} \label{subsec:existence-skeleton}

For any $h\in L^2(0,T; \mathcal{H})$, there exists a sequence $h_n\in L^\infty(0,T; \mathcal{H})$ such that
\begin{align}
h_n \goto h~~~\text{in}~~L^2(0,T; \mathcal{H})~~~\text{as}~~~n\goto \infty\,. \label{conv:control-for-existence-skeleton}
\end{align}
In view of Subsection \ref{subsec:existence-skeleton-bounded-control}, there exists a unique solution $u_n\in C([0,T];L^2(\R^d))\cap L^q(0,T; \mathcal{Y})$ to \eqref{eq:skeleton} with $h$ replaced by $h_n$ i.e., $u_n(\cdot)$ satisfies the following equation:
\begin{align}
u_n(t)- u_0 - \int_0^t {\rm div}_x \big( |\nabla u_n|^{p-2}\nabla u_n\big)\,ds =  \int_0^t \sigma(u_n(s))h_n(s)\,ds\,. \label{eq:skeleton-n-general}
\end{align}

By using chain-rule, the Cauchy-Schwartz inequality, Young's inequality and the assumptions \ref{A2}-\ref{A3}, one has 
\begin{align}
&\|u_n(t)\|_{L^2(\R^d)}^2 + 2 \int_0^t \|\nabla u_n(s)\|_{L^p(\R^d)^d}^p\,ds \notag \\
& \le \|u_0\|_{L^2(\R^d)}^2 + 2\int_0^t \|\sigma(u_n)\|_{\mathcal{L}_2(\mathcal{H}, L^2(\R^d))}\|h_n(s)\|_{\mathcal{H}} \|u_n(s)\|_{L^2(\R^d)}\,ds \notag \\
&\le  \Big( \|u_0\|_{L^2(\R^d)}^2 + 2 \sigma_b^2 \int_0^T \|h_n(s)\|_{\mathcal{H}}^2\,ds \Big) + \int_0^t \big( 1 + 2 \sigma_b^2 \|h_n(s)\|_{\mathcal{H}}^2\big)\|u_n(s)\|_{L^2(\R^d)}^2\,ds\,. \label{inq-1-skeleton-n-general}
\end{align}
By Gronwall's lemma, we get
\begin{align}
\sup_{t\in [0,T]}\|u_n(t)\|_{L^2(\R^d)}^2 \le C \Big( \|u_0\|_{L^2(\R^d)}^2 + 2 \sigma_b^2 \int_0^T \|h_n(s)\|_{\mathcal{H}}^2\,ds \Big)\exp\Big( \int_0^T\|h_n(s)\|_{\mathcal{H}}^2\,ds \Big)\,.  \label{inq-2-skeleton-n-general}
\end{align}
In view of \eqref{conv:control-for-existence-skeleton}, there exists $\bar{M}>0$, independent of $n$, such that
\begin{align}
\int_0^T \|h_n(s)\|_{\mathcal{H}}^2\,ds \le \bar{M}\,. \label{inq:boundeness-control}
\end{align}
Using \eqref{inq:boundeness-control} in \eqref{inq-2-skeleton-n-general}, we see that there exists a constant $C>0$, independent of $n$ such that 
\begin{align}
\sup_{t\in [0,T]}\|u_n(t)\|_{L^2(\R^d)}^2 \le C \,. \label{inq:uni-l2-skeleton-general}
\end{align}
Again, we use \eqref{inq:uni-l2-skeleton-general} in \eqref{inq-1-skeleton-n-general} to obtain
\begin{align}
\int_0^T \|\nabla u_n(s)\|_{L^p(\R^d)^d}^p\,ds\le C\,. \label{inq:uni-grad-lp-skeleton-general}
\end{align}
Observe that thanks to \ref{A3}, the uniform bound \eqref{inq:uni-l2-skeleton-general} and \eqref{inq:boundeness-control}, one has,  for any $n, m \in \mathbb{N}$, 
\begin{align}
& \|u_n(t)-u_m(t)\|_{L^2(\R^d)}^2  \notag \\
& \le 2 \int_0^t \big\langle (\sigma(u_n)-\sigma(u_m)) h_n, u_n(s)-u_m(s)\big\rangle\,ds 
+ 2 \int_0^t \big\langle \sigma(u_m)(h_n-h_m), u_n(s)-u_m(s)\big\rangle\,ds \notag \\
& \le  \int_0^t \big( 1 + c_\sigma^2\|h_n(s)\|_{\mathcal{H}}^2\big) \|u_n(s)-u_m(s)\|_{L^2(\R^d)}^2\,ds + \int_0^t \|u_n-u_m\|_{L^2(\R^d)}^2\,ds \notag \\
& \hspace{1cm}+ \int_0^t \|\sigma(u_m)\|_{\mathcal{L}_2(\mathcal{H}, L^2(\R^d))}^2 \|h_n(s)-h_m(s)\|_{\mathcal{H}}^2\,ds  \notag \\
& \le \int_0^t \big( 2 +  c_\sigma^2\|h_n(s)\|_{\mathcal{H}}^2\big) \|u_n(s)-u_m(s)\|_{L^2(\R^d)}^2\,ds + 2 \sigma_b^2 \int_0^t \big( 1 + \|u_m(s)\|_{L^2(\R^d)}^2\big) \|h_n(s)-h_m(s)\|_{\mathcal{H}}^2\,ds \notag \\
& \le \int_0^t \big( 2 +  c_\sigma^2\|h_n(s)\|_{\mathcal{H}}^2\big) \|u_n(s)-u_m(s)\|_{L^2(\R^d)}^2\,ds + C \int_0^T\|h_n(s)-h_m(s)\|_{\mathcal{H}}^2\,ds\,, \notag
\end{align}
and hence, by Gronwall's lemma
\begin{align}
\sup_{t\in [0,T]} \|u_n(t)-u_m(t)\|_{L^2(\R^d)}^2 & \le  C \exp\Big\{ \int_0^T \|h_n(s)\|_{\mathcal{H}}^2\,ds\Big\}
\int_0^T \|h_n(s)-h_m(s)\|_{\mathcal{H}}^2\,ds  \notag \\
&\le C \int_0^T\|h_n(s)-h_m(s)\|_{\mathcal{H}}^2\,ds\,, \notag 
\end{align}
for some constant $C>0$, independent of $n$ and $m$. This show that $\{u_n\}$ is a Cauchy sequence in $C([0,T]; L^2(\R^d))$. Using this and
\eqref{inq:uni-grad-lp-skeleton-general}, one may extract a subsequence of $\{u_n\}$, still we denoted by $\{u_n\}$, such that as $n\goto \infty$
\begin{equation}\label{conv:skeleton-n-general}
u_n \goto \bar{u}_*~~~\text{in}~~C([0,T];L^2(\R^d)); \quad 
\nabla u_n \rightharpoonup \nabla \bar{u}_* ~~~\text{in}~~L^p(0,T;L^p(\R^d)^d)
\end{equation}
for some $\bar{u}_* \in C([0,T]; L^2(\R^d))\cap L^q(0,T; \mathcal{Y})$. Using \eqref{conv:skeleton-n-general} and \eqref{conv:control-for-existence-skeleton}, and following the argument as invoked in Subsection \ref{subsec:existence-skeleton-bounded-control}, one can pass to the limit in equation \eqref{eq:skeleton-n-general} as $n\goto \infty$ to obtain
\begin{align*}
\bar{u}_*(t)-u_0 - \int_0^t {\rm div}_x \big( |\nabla \bar{u}_*|^{p-2}\nabla \bar{u}_* \big)\,ds= \int_0^t f(\bar{u}_*(s))\,ds + \int_0^t \sigma(\bar{u}_*(s))h(s)\,ds \notag
\end{align*}
for all $t\in [0,T]$. In other words, $\bar{u}_*$ is the unique solution of the skeleton equation \eqref{eq:skeleton}. 
\subsection{Proof of Theorem \ref{thm:skeleton}}
Existence of a unique solution $u_h$ for the skeleton equation \eqref{eq:skeleton} follows from Subsection \ref{subsec:existence-skeleton}.
Let $ h\in S_N$. Then $h\in L^2(0,T;\mathcal{H})$ and $\int_0^T \|h(s)\|_{\mathcal{H}}^2\,ds \le N$. Since $u_h$ is the solution of  \eqref{eq:skeleton}, by using chain rule, the Cauchy-Schwartz inequality, Young's inequality and the assumptions \ref{A2}-\ref{A3}, we have 
\begin{align*}
& \|u_h(t)\|_{L^2(D)}^2 + \int_0^t \|\nabla u_h(s)\|_{L^2(\R^d)^d}^p\,ds \notag \\
& \le  \Big( \|u_0\|_{L^2(\R^d)}^2 + 2 \sigma_b^2 \int_0^T \|h(s)\|_{\mathcal{H}}^2\,ds \Big) + \int_0^t \big( 1 + 2 \sigma_b^2 \|h(s)\|_{\mathcal{H}}^2\big)\|u_h(s)\|_{L^2(\R^d)}^2\,ds\,.
\end{align*}
An application of Gronwall's lemma along with the fact that $\int_0^T \|h(s)\|_{\mathcal{H}}^2\,ds \le N$, we get the uniform estimate 
\begin{align}
\sup_{h\in S_N} \Big\{ \sup_{t\in [0,T]} \|u_h(t)\|_{L^2(\R^d)}^2 + \int_0^T \|u_h(t)\|_{\mathcal{Y}}^q\,dt \Big\} \le C_N\, \notag
\end{align}
for some constant $C_N>0$, where $q=\max\{2, p\}$. This completes the proof of Theorem \ref{thm:skeleton}.

\section{ Large deviation principle: Proof of Theorem \ref{thm:ldp}}\label{sec:proof-ldp}
In this section, we give the proof of Theorem \ref{thm:ldp}. According to \cite[Theorem 3.2]{MSZ}, which are based on the result of Budhiraja and Dupuis in \cite[Theorem 4.4]{budhi}, it is sufficient  to prove the following two conditions:
\begin{Assumptions2}
\item \label{C1} 
 For every $M < \infty$, for any family  $\{h_\epsilon; \epsilon > 0\} \subset \mathcal{A}_M$, and any $\gamma>0$, there holds
\begin{align*}
\lim_{\epsilon \goto 0} \mathbb{P}\Big( \rho(v_\eps, \bar{u}_\eps)> \gamma \Big)=0\,,
\end{align*}
where $\rho(\cdot, \cdot)$ is  a given metric on the Polish space $C([0,T]; L^2(\R^d))$ given by 
$$ \rho^2(u,v)= \sup_{s\in [0,T]} \|u(s)-v(s)\|_{L^2(\R^d)}^2,\quad \forall~ u, v \in C([0,T]; L^2(\R^d),$$
and 
\begin{align*}
   v_\eps(\cdot):= \mathcal{G}_{\epsilon} \bigg( W(\cdot) + \frac{1}{{\epsilon}} \int_0^{\cdot} h^{\epsilon}(s)\,{\rm d}s\bigg),\quad 
   \bar{u}_\eps(\cdot):= \mathcal{G}_0 \bigg( \int_0^{\cdot}h_\eps(s) \, {\rm d}s\bigg)\,.
\end{align*}
\item \label{C2} For every  $M < \infty$, the set
\begin{align*} 
    K_{M} = \left \{\mathcal{G}_0\bigg(\int_0^{.}h(s)\,{\rm d}s\bigg) : h \in S_{M} \  \right \}
\end{align*}
is a compact subset of $\mathcal{Z}=C([0,T]; L^2(\R^d)$.
\end{Assumptions2}
In the following subsections, we prove this two conditions. 
\subsection{ Proof of condition \ref{C1}}
For any fixed $M< + \infty$, and any family $\{ h_\eps\}\subset \mathcal{A}_M$, one can use Girsanov's theorem and the uniqueness of solutions of \eqref{eq:ldp} to infer that 
$v_\eps(t):=\mathcal{G}_\eps\big( W(t) + \frac{1}{\eps}\int_0^t h_\eps(s)\,ds\big)$ is the unique solution of the SPDE
\begin{equation}\label{eq:ldp-1}
\begin{aligned}
& dv_\eps- {\rm div}_x (|\nabla v_\eps|^{p-2} \nabla v_\eps)\,dt= \sigma(v_\eps)h_\eps(t)\,dt + \eps\, \sigma(v_\eps)\,dW(t) \\
& v_\eps(0,x)=u_0(x),~~x\in \mathbb{R}^d\,,
\end{aligned}
\end{equation}
Moreover, $v_\eps \in L^2(\Omega; C([0,T];L^2(\R^d)))\cap L^q(\Omega\times(0,T); \mathcal{Y})$. Again for any fixed $M>0$, let $\bar{u}_\eps$ be the unique solution of the skeleton equation \eqref{eq:skeleton} with $h$ replaced by $h_\eps \in \mathcal{A}_M$. Then according to the definition of the mapping $\mathcal{G}_0$, one has $\bar{u}_\eps(t)=\mathcal{G}_0 \big( \int_0^t h_\eps(s)\,ds\Big)$.  Moreover, thanks to \eqref{esti:uniform-skeleton-original}, there exists a constant $C_M>0$, independent of $\epsilon$, such that  $\mathbb{P}$-a.s., 
\begin{align}
\sup_{\eps >0}\Big\{ \sup_{s\in [0,T]}\|\bar{u}_\eps(s)\|_{L^2(\R^d)}^2\Big] + \int_0^T \|\bar{u}_\eps(s)\|_{\mathcal{Y}}^q\,ds \Big\}\le C_M\,.  \label{esti:uni-bar-u-eps}
\end{align}
\subsubsection{\bf Uniform estimate for $v_\eps$:}
We first deduce the uniform estimate for $v_\eps$ of equation \eqref{eq:ldp-1} . An application of It\^{o} formula to $x\mapsto \|x\|_{\R^d}^2$ along with assumption \ref{A2}, and the Cauchy-Schwartz inequality, we have
\begin{align}
& \|v_\eps(t)\|_{L^2(\R^d)}^2  + 2 \int_0^t \|\nabla v_\eps(s)\|_{L^p(\R^d)^d}^p\,ds \notag  \\
& \le \|u_0\|_{L^2(\R^d)}^2  + 2 \int_0^t \Big| \langle \sigma(v_\eps)h_\eps(s), v_\eps(s)\rangle \Big|\,ds + 2 \eps\, \Big| \int_0^t \langle v_\eps(s), \sigma(v_\eps(s))\,dW(s)\rangle\Big|\,. \label{inq:uni-1}
\end{align}
Discarding the second non-negative term of the left hand side of the above inequality, we get, after taking the expectation, 
\begin{align*}
\mathbb{E}\Big[ \sup_{s\in [0,t]}\|v_\eps(s)\|_{L^2(\R^d)}^2\Big] & \le \|u_0\|_{L^2(\R^d)}^2 + 2\mathbb{E}\Big[\int_0^t \Big| \langle \sigma(v_\eps)h_\eps(s), v_\eps(s)\rangle \Big|\,ds\Big] \\
& \hspace{2cm}+ 
2 \eps\, \mathbb{E}\Big[ \sup_{s\in [0,t]} \Big| \int_0^s \langle v_\eps(r), \sigma(v_\eps)\,dW(r)\rangle\Big|\Big]= \sum_{i=1}^3 {\tt A}_i\,.
\end{align*}
By using the Cauchy-Schwartz and Young's inequalities together with the fact that $h_\eps \in \mathcal{A}_M$, we have 
\begin{align}
{\tt A}_2&  \le \frac{1}{4} \mathbb{E}\Big[ \sup_{s\in [0,t]}\|v_\eps(s)\|_{L^2(\R^d)}^2\Big] + 
C \mathbb{E}\Big[ \Big( \int_0^t \|\sigma(v_\eps(s))\|_{\mathcal{L}_2(\mathcal{H}, L^2(\R^d))}\|h_\eps(s)\|_{\mathcal{H}}\,ds\Big)^2\Big] \notag \\
& \le \frac{1}{4} \mathbb{E}\Big[ \sup_{s\in [0,t]}\|v_\eps(s)\|_{L^2(\R^d)}^2\Big] + C  \mathbb{E}\Big[ \Big( \int_0^t (1 + \|v_\eps(s)\|_{L^2(\R^d)}^2)\,ds \Big) \int_0^T\|h_\eps(s)\|_{\mathcal{H}}^2\,ds\Big] \notag \\
& \le \frac{1}{4} \mathbb{E}\Big[ \sup_{s\in [0,t]}\|v_\eps(s)\|_{L^2(\R^d)}^2\Big]  + C_M \mathbb{E}\Big[\int_0^t \big(1 + \|v_\eps(s)\|_{L^2(\R^d)}^2\big)\,ds\Big]\,. \label{esti:a2}
\end{align}
In view of Burkholder-Davis-Gundy inequality~(BDG inequality in short) and the assumption \ref{A3}, we see that
\begin{align*}
{\tt A}_3 & \le C \eps \Big\{ \mathbb{E}\Big[ \int_0^t \|\sigma(v_\eps)\|_{\mathcal{L}_2(\mathcal{H}, L^2(\R^d))}^2 \|v_\eps(s)\|_{L^2(\R^d)}^2\,ds\Big]\Big\}^\frac{1}{2} \\
& \le \frac{1}{4} \mathbb{E}\Big[ \sup_{s\in [0,t]}\|v_\eps(s)\|_{L^2(\R^d)}^2\Big] + C_\eps \mathbb{E}\Big[\int_0^t \big(1 + \|v_\eps(s)\|_{L^2(\R^d)}^2\big)\,ds\Big]\,.
\end{align*}
Hence, we have, for all $\eps \in (0,1]$,
\begin{align*}
\mathbb{E}\Big[ \sup_{s\in [0,t]}\|v_\eps(s)\|_{L^2(\R^d)}^2\Big] \le 2 \|u_0\|_{L^2(\R^d)}^2 + C_{M,T} + 
C \int_0^t \mathbb{E}\Big[ \sup_{r\in [0,s]}\|v_\eps(r)\|_{L^2(\R^d)}^2\Big]\,dr\,.
\end{align*}
By Gronwall's lemma, there exists a constant $C>0$, independent of $\eps>0$, such that 
\begin{align}
\mathbb{E}\Big[ \sup_{s\in [0,T]}\|v_\eps(s)\|_{L^2(\R^d)}^2\Big]\le C\,. \label{uni-v-eps}
\end{align}
Using the estimates \eqref{esti:a2} and \eqref{uni-v-eps} in \eqref{inq:uni-1}, one can easily get a constant $C>0$, independent of $\eps$, such that
\begin{align}
\sup_{\eps >0}\Big\{ \mathbb{E}\Big[ \sup_{s\in [0,T]}\|v_\eps(s)\|_{L^2(\R^d)}^2\Big] + \mathbb{E}\Big[ \int_0^T \|v_\eps(s)\|_{\mathcal{Y}}^q\,ds \Big]\Big\}\le C\,.\notag 
\end{align}

\subsubsection{\bf Verification of condition \ref{C1}}
Note that the metric $\rho$ on $\mathcal{Z}:=C[(0,T]; L^2(\R^d))$ is given by
\begin{align*}
\rho(u,v):= \Big\{ \sup_{t\in [0,T]}\|u(t)-v(t)\|_{L^2(\R^d)}^2\Big\}^\frac{1}{2},\quad \forall~u,v\in \mathcal{Z}\,.
\end{align*}
Thanks to  Markov inequality, to validate the condition \ref{C1}, it suffices to show that
\begin{align}
\lim_{\epsilon \goto 0} \mathbb{E}\Big[ \sup_{t\in [0,T]}\|v_\eps(t)-\bar{u}_\epsilon(t)\|_{L^2(\R^d)}^2\Big]=0\,. \label{cond:c1-sufficient}
\end{align}
Let $z_\eps(t)= v_\eps(t)-\bar{u}_\eps(t)$. Then $z_\eps(\cdot)$ satisfies the following SPDE
\begin{equation*}
\begin{aligned}
& dz_\eps(t)-{\rm div}_x \Big( |\nabla v_\eps|^{p-2}\nabla v_\eps - |\nabla \bar{u}_\eps|^{p-2}\nabla \bar{u}_\eps\Big)\,dt \\
& \hspace{2cm}=\big( \sigma(v_\eps)-\sigma(\bar{u}_\eps)\big) h_\eps(t)\,dt + \eps \sigma(v_\eps)\,dW(t),~~~(t,x)\in (0,T)\times \R^d\,, \\
&z_\eps(0,x)=0,~~~x\in \R^d\,.
\end{aligned}
\end{equation*}
By using It\^{o} formula, the Cauchy-Schwartz inequality, Young's inequality, assumptions \ref{A2}-\ref{A3}, the uniform estimate \eqref{uni-v-eps}, and the inequality \eqref{inq:monotone}, we have 
\begin{align}
\|z_\eps(t)\|_{L^2(\R^d)}^2 & \le C \int_0^t \| \sigma(v_\eps(s))-\sigma(\bar{u}_\eps)\|_{\mathcal{L}_2(\mathcal{H}, L^2(\R^d))}^2 \|h_\eps(s)\|_{\mathcal{H}}^2\,ds \notag \\
& + \eps^2 \int_0^t \|\sigma(v_\eps(s))\|_{\mathcal{L}_2(\mathcal{H}, L^2(\R^d))}^2\,ds 
 + 2 \eps \sup_{s\in [0,t]} \Big| \int_0^s \langle z_\eps(r), \sigma(v_\eps(r))\,dW(r)\rangle\Big| \notag \\
& \le C T \eps^2 + C \int_0^t  \|h_\eps(s)\|_{\mathcal{H}}^2 \|z_\eps(s)\|_{L^2(\R^d)}^2\,ds + 2 \eps\, \sup_{s\in [0,t]} \Big| \int_0^s \langle z_\eps(r), \sigma(v_\eps(r))\,dW(r)\rangle\Big|\,. \label{inq-1-c1}
\end{align}
An application of Gronwall's lemma on \eqref{inq-1-c1} then gives 
\begin{align}
\sup_{t\in [0, T]}\|z_\eps(t)\|_{L^2(\R^d)}^2 & \le \Bigg\{ CT \eps^2 + 2 \eps \sup_{s\in [0, T]} \Big| \int_0^s \langle z_\eps(r), \sigma(v_\eps(r))\,dW(r)\rangle\Big|\Bigg\} \exp\Big( \int_0^T  \|h_\eps(s)\|_{\mathcal{H}}^2\,ds \Big)\,. \notag 
\end{align}
Since $h_\eps \in \mathcal{A}_M$, we get, after taking expectation,
\begin{align}
& \mathbb{E}\Big[\sup_{t\in [0, T]}\|z_\eps(t)\|_{L^2(\R^d)}^2\Big] \notag \\
& \le 
C(M,T) \eps^2 
+ C(M,T)\, \eps \,\mathbb{E}\Big[ \sup_{s\in [0, T]} \Big| \int_0^s \langle z_\eps(r), \sigma(v_\eps(r))\,dW(r)\rangle\Big|\Big]\equiv \sum_{i=1}^2 {\tt C}_i\,.\label{inq-2-c1}
\end{align}
Thanks to BDG inequality, Young's inequality, \eqref{inq: growth-sigma}, and \eqref{uni-v-eps}, we see that
\begin{align}
{\tt C}_2 & \le C(M,T)\eps\, \mathbb{E}\Big[\Big( \int_0^{T}\|z_\eps(r)\|_{L^2(\R^d)}^2\|\sigma(v_\eps(r))\|_{\mathcal{L}_2(\mathcal{H}, L^2(\R^d))}^2\,dr\Big)^\frac{1}{2}\Big] \notag \\
& \le  \frac{1}{4} \mathbb{E}\Big[\sup_{t\in [0, T]}\|z_\eps(t)\|_{L^2(\R^d)}^2\Big]  + C(M,T)\,\eps^2\, \mathbb{E}\Big[ \int_0^T \big( 1 + \|v_\eps(r)\|_{L^2(\R^d)}^2\big)\,dr\Big] \notag \\
& \le \frac{1}{4} \mathbb{E}\Big[\sup_{t\in [0, T]}\|z_\eps(t)\|_{L^2(\R^d)}^2\Big] + C(T, M)\eps^2\,.\label{esti:c2}
\end{align}
Hence the assertion \eqref{cond:c1-sufficient} holds once we use \eqref{esti:c2} in \eqref{inq-2-c1} and pass to the limit as $\eps \goto 0$. 

\subsection{Proof of condition \ref{C2}}\label{subsec:cond-c2}
For any fixed $M>0$, 
let $\{\bar{h}_n\} \subset S_M$ be a sequence of functions such that $\bar{h}_n \goto \bar{h}$ weakly in $L^2(0,T; \mathcal{H})$. Since $S_M$ is weakly compact, to prove condition \ref{C2}, it suffices to show that $\bar{u}_n \goto \bar{u}_{\bar{h}}$ strongly in $C([0,T]; L^2(\R^d))$, where $\bar{u}_n$ and $\bar{u}_{\bar{h}}$ are the solutions of \eqref{eq:skeleton} corresponding to $\bar{h}_n$ and $\bar{h}$ respectively. According to  \eqref{esti:uniform-skeleton-original},  there exists a constant ${\tt K}$, independent of $n$, such that 
 \begin{align}
\sup_{n\in \mathbb{N}} \Big\{ \sup_{t\in [0,T]} \big( \|\bar{u}_n(t)\|_{L^2(\R^d)}^2 + \|\bar{u}_{\bar{h}}(t)\|_{L^2(\R^d)}^2\big) + \int_0^T 
\big(\|\bar{u}_n(t)\|_{\mathcal{Y}}^q + \|\bar{u}_{\bar{h}}(t)\|_{\mathcal{Y}}^q\big)\,dt \Big\}={\tt K}\,. \label{apriori-cond-c2}
 \end{align}
 As mentioned earlier, due to the lack of compactness of the embedding $\mathcal{Y}\hookrightarrow L^2(\R^d) \hookrightarrow \mathcal{Y}^\prime$, it prevent us to show that strong  convergence of $\{\bar{u}_n\}$ in a straight-forward manner.  To overcome this, we deduce the  localized time increment estimate. For any given small $\delta >0$, let $t(\delta):=\lfloor \frac{t}{\delta}\rfloor \delta$, where 
$\lfloor s \rfloor$ denotes the largest integer smaller or equal to $s$. 
\begin{lem}\label{lem:time-increment-un-cond-c2}
Let $\bar{u}_n$ be the unique solution of \eqref{eq:skeleton} with $\bar{h}_n\in S_M$. Then, there exists a positive constant $C$, independent of $n$, such that 
\begin{align}
\int_0^{T} \| \bar{u}_n(t)-\bar{u}_n(t(\delta))\|_{L^2(\R^d)}^2\,dt\le C\delta\,. \label{esti:time-increment-un-cond-2}
\end{align}
\end{lem}
\begin{proof}
Observe that, thanks to \eqref{apriori-cond-c2} 
\begin{align}
&\int_0^{T} \| \bar{u}_n(t)-\bar{u}_n(t(\delta))\|_{L^2(\R^d)}^2\,dt  \notag \\
& = \int_0^{\delta} \| \bar{u}_n(t)-u_0\|_{L^2(\R^d)}^2\,dt
+ \int_\delta^{T} \| \bar{u}_n(t)-\bar{u}_n(t(\delta))\|_{L^2(\R^d)}^2\,dt \notag \\
&  \le C \delta + 2 \int_\delta^{T} \| \bar{u}_n(t)-\bar{u}_n(t-\delta)\|_{L^2(\R^d)}^2\,dt + 2  \int_\delta^{T} \| \bar{u}_n(t-\delta)-\bar{u}_n(t(\delta))\|_{L^2(\R^d)}^2\,dt \notag \\
&\equiv C \delta + \sum_{i=1}^2{\pmb B}_i\,. \label{inq-1-time-increment-un-cond-c2}
\end{align}
We first estimate the term ${\pmb B}_1$. Thanks to Cauchy-Schwartz inequality, the assumption \ref{A2}, we have
\begin{align}
{\pmb B}_1 &\le 2 \int_{\delta}^T \int_{t-\delta}^t \|\Delta_p (\bar{u}_n(s))\|_{\mathcal{Y}^\prime}\|\bar{u}_n(s)-\bar{u}_n(t-\delta)\|_{\mathcal{Y}}\,ds\,dt \notag \\
& + C  \int_{\delta}^T \int_{t-\delta}^t \|\bar{u}_n(s)\|_{L^2(\R^d)} \|\bar{u}_n(s)-\bar{u}_n(t-\delta)\|_{L^2(\R^d)}\,ds\,dt \notag \\
& + C  \int_{\delta}^T \int_{t-\delta}^t \|\sigma(\bar{u}_n(s))\|_{\mathcal{L}_2(\mathcal{H}, L^2(\R^d))}\|\bar{h}_n(s)\|_{\mathcal{H}} \|\bar{u}_n(s)-\bar{u}_n(t-\delta)\|_{L^2(\R^d)}\,ds\,dt \equiv \sum_{i=1}^3 {\pmb B}_{1,i}\,. \notag 
\end{align}
By H\"{o}lder's inequality and the uniform estimate \eqref{apriori-cond-c2}, we have
\begin{equation*}
\begin{aligned}
{\pmb B}_{1,1} & \le C \Big\{  \int_{\delta}^T \int_{t-\delta}^t \|\Delta_p (\bar{u}_n(s))\|_{\mathcal{Y}^\prime}^{\frac{q}{q-1}}\,ds\,dt \Big\}^\frac{q-1}{q} \Big\{  \int_{\delta}^T \int_{t-\delta}^t \|\bar{u}_n(s)-\bar{u}_n(t-\delta)\|_{\mathcal{Y}}^q\,ds\,dt\Big\}^\frac{1}{q} \\
& \le C \delta  \Big\{  \int_{0}^T \|\Delta_p (\bar{u}_n(s))\|_{\mathcal{Y}^\prime}^{\frac{q}{q-1}}\,ds\,dt \Big\}^\frac{q-1}{q} \Big\{ \int_0^T \|\bar{u}_n(t)\|_{\mathcal{Y}}^q\,dt\Big\}^\frac{1}{q} \\
& \le  C \delta  \Big\{  \int_{0}^T \big( 1 + \|\nabla \bar{u}_n(s)\|_{L^p(\R^d)^d}^q\big)\,ds\, \Big\}^\frac{q-1}{q} \le C \delta\,, \\
{\pmb B}_{1,2} & \le C \Big\{  \int_{\delta}^T \int_{t-\delta}^t \|\bar{u}_n(s)\|_{L^2(\R^d)}^2\,ds\,dt \Big\}^\frac{1}{2} 
\Big\{  \int_{\delta}^T \int_{t-\delta}^t \|\bar{u}_n(s)-\bar{u}_n(t-\delta)\|_{L^2(\R^d)}^2\,ds\,dt \Big\}^\frac{1}{2} \le C \delta\,, \\
{\pmb B}_{1,3} & \le  C  \Big\{  \int_{\delta}^T \int_{t-\delta}^t \big( 1 +\|\bar{u}_n(s)\|_{L^2(\R^d)}^2\big)\|\bar{h}_n(s)\|_{\mathcal{H}}^2\,ds\,dt \Big\}^\frac{1}{2} \Big\{  \int_{\delta}^T \int_{t-\delta}^t \|\bar{u}_n(s)-\bar{u}_n(t-\delta)\|_{L^2(\R^d)}^2\,ds\,dt \Big\}^\frac{1}{2} \\
& \le C \sqrt{\delta} \Big\{  \int_{\delta}^T \int_{t-\delta}^t \|\bar{h}_n(s)\|_{\mathcal{H}}^2\,ds\,dt \Big\}^\frac{1}{2} \le C \delta\,.
\end{aligned}
\end{equation*}
Thus, there exists a constant $C>0$, independent of $n$ such that
\begin{align}
{\pmb B}_1 \le C \delta\,. \label{esti:b1-cond-c2}
\end{align}
Following the similar argument as invoked in the proof of \eqref{esti:b1-cond-c2}, we infer that, there exists a constant $C>0$, only depending on $T, q, M, u_0, {\tt K}$, such that
\begin{align}
{\pmb B}_2 \le C \delta\,. \label{esti:b2-cond-c2}
\end{align}
Putting the inequalities \eqref{esti:b1-cond-c2} and \eqref{esti:b2-cond-c2} in \eqref{inq-1-time-increment-un-cond-c2}, we arrive at the assertion \eqref{esti:time-increment-un-cond-2}. 
\end{proof}
\begin{rem}\label{rem:time-increment-uh-cond-c2}
In view of the proof of Lemma \ref{lem:time-increment-un-cond-c2}, one can easily see that $\bar{u}_{\bar{h}}$ satisfies the estimate
\eqref{esti:time-increment-un-cond-2}. 
\end{rem}
\subsubsection{\bf Validation of condition \ref{C2}:} To validate the condition \ref{C2}, we need to show that
\begin{align}
\lim_{n\goto \infty} \sup_{t\in [0,T]}\| \bar{u}_n(t)-\bar{u}_{\bar{h}}(t)\|_{L^2(\R^d)}^2=0\,. \label{cond:c2-sufficient}
\end{align}
Let $X_n(t)= \bar{u}_n(t)-\bar{u}_{\bar{h}}(t)$. Then $X_n(t)$ satisfies the following deterministic PDE
\begin{equation*}
\begin{aligned}
&\frac{d}{dt} X_n(t) - {\rm div}_x \Big( |\nabla \bar{u}_n|^{p-2}\nabla \bar{u}_n - 
|\nabla \bar{u}_{\bar{h}}|^{p-2}\nabla \bar{u}_{\bar{h}}\Big)=  \big( \sigma(\bar{u}_n)\bar{h}_n(t)- 
\sigma(\bar{u}_{\bar{h}})\bar{h}(t)\big), \\
&X_n(0)=0\,.
\end{aligned}
\end{equation*}
An application of chain-rule, the Cauchy-Schwartz inequality, the assumptions \ref{A2} and \ref{A3}, Young's inequality, and the inequality \eqref{inq:monotone} gives
\begin{align*}
\|X_n(t)\|_{L^2(\R^d)}^2 \le  2 \int_0^T \big| \langle 
\sigma(\bar{u}_{\bar{h}}(t))(\bar{h}_n(t)-\bar{h}(t)), X_n(t)\rangle \big|\,dt\,.
\end{align*}
Thus, using the fact that $\bar{h}_n \in S_M$, we have
\begin{align}
\sup_{t\in [0,T]}\|X_n(t)\|_{L^2(\R^d)}^2 & \le 2 \int_0^T\Big| \big\langle \sigma(\bar{u}_{\bar{h}}(s))(\bar{h}_n(s)-\bar{h}(s)), X_n(s)\big\rangle\Big|\,ds \notag \\
& \le C \int_0^T \Big| \langle \sigma(\bar{u}_{\bar{h}}(s))(\bar{h}_n(s)-\bar{h}(s)), X_n(s)- X_n(s(\delta))\rangle \Big|\,ds \notag \\
& + C \int_0^T \Big| \langle \big\{\sigma(\bar{u}_{\bar{h}}(s))- \sigma(\bar{u}_{\bar{h}}(s(\delta)))\big\}(\bar{h}_n(s)-\bar{h}(s)), X_n(s(\delta))\rangle \Big|\,ds \notag \\
& + C \sup_{t\in [0,T]} \Big| \int_{t(\delta)}^t \sigma(\bar{u}_{\bar{h}}(s(\delta)))(\bar{h}_n(s)-\bar{h}(s)), X_n(s(\delta))\rangle \,ds\Big| \notag \\
& + C \sup_{t\in [0,T]} \sum_{k=0}^{\lfloor \frac{t}{\delta}\rfloor -1} \Big| \langle \sigma(\bar{u}_{\bar{h}}(k\delta) \int_{k\delta}^{(k+1)\delta} (\bar{h}_n(s)-\bar{h}(s))\,ds, X_n(k\delta)\rangle \Big|\equiv \sum_{i=1}^4 {\pmb C}_{i,n}\,.
\label{inq:cond-c2-1}
\end{align}
 In view of Lemma \ref{lem:time-increment-un-cond-c2} and Remark \ref{rem:time-increment-uh-cond-c2}, we get, 
 \begin{align}
\int_0^T \|X_n(t)- X_n(t(\delta))\|_{L^2(\R^d)}^2\,dt \le C \delta\,. \label{inq:time-increment-difference-cond-c2}
 \end{align}
 By using H\"{o}lder's inequality, Remark \ref{rem:time-increment-uh-cond-c2}, \eqref{apriori-cond-c2}, \eqref{inq:time-increment-difference-cond-c2}, the assumption \ref{A3}, and the fact that $\bar{h}_n, \bar{h}\in S_M$, we get
 \begin{equation}\label{esti:c1n-c3n}
\begin{aligned}
{\pmb C}_{1,n} & \le C \Big( \int_0^T \|\sigma(\bar{u}_{\bar{h}}(s))\|_{\mathcal{L}_2(\mathcal{H}, L^2(\R^d))}^2 \|\bar{h}_n(s)-\bar{h}(s)\|_{\mathcal{H}}^2\,ds \Big)^\frac{1}{2} \Big( \int_0^T \| X_n(s)- X_n(s(\delta))\|_{L^2(\R^d)}^2\,ds\Big)^\frac{1}{2} \\
& \le C \sqrt{\delta} \Big( \int_0^T (1+ \| \bar{u}_{\bar{h}}(s)\|_{L^2(\R^d)}^2) \|\bar{h}_n(s)-\bar{h}(s)\|_{\mathcal{H}}^2\,ds \Big)^\frac{1}{2} \le C \sqrt{\delta}\,, \\
{\pmb C}_{2,n} & \le C \Big( \int_0^T \| \bar{u}_{\bar{h}}(s)-\bar{u}_{\bar{h}}(s(\delta))\|_{L^2(\R^d)}^2\,ds \Big)^{\frac{1}{2}} \Big( \int_0^T \|\bar{h}_n(s)-\bar{h}(s)\|_{\mathcal{H}}^2 \|X_n(s(\delta))\|_{L^2(\R^d)}^2\,ds\Big)^\frac{1}{2} \\
& \le C\sqrt{\delta} \Big( \int_0^T \|\bar{h}_n(s)-\bar{h}(s)\|_{\mathcal{H}}^2 \,ds\Big)^\frac{1}{2} \le C \sqrt{\delta}\,, \\
{\pmb C}_{3,n} & \le C \sqrt{\delta} \Big( \sup_{t\in [0,T]}\|X_n(t)\|_{L^2(\R^d)}^2\Big)^\frac{1}{2} \Big( \int_0^T \big( 1 + \|\bar{u}_{\bar{h}}(s(\delta))\|_{L^2(\R^d)}^2\big) \|\bar{h}_n(s)-\bar{h}(s)\|_{\mathcal{H}}^2 \,ds\Big)^\frac{1}{2} \le C \sqrt{\delta}\,.
\end{aligned}
\end{equation}
 Since $\sigma$ is an operator from $L^2(\R^d) $ to $\mathcal{L}_2(\mathcal{H}, L^2(\R^d))$, we see that $\sigma(\bar{u}_{\bar{h}}(k\delta)$ is a compact operator on $L^2(\R^d)$, and hence one has, for each fixed $k$,
 \begin{align*}
\Big\| \sigma(\bar{u}_{\bar{h}}(k\delta) \int_{k\delta}^{(k+1)\delta} (\bar{h}_n(s)-\bar{h}(s))\,ds\Big\|_{L^2(\R^d)}\goto 0~~~\text{as}~~n\goto \infty\,.
 \end{align*}
 Moreover, it is easy to see that $|{\pmb C}_{4,n}| \le C$ for some constant $C>0$ independent of $n$. Hence we conclude that
 \begin{align}
{\pmb C}_{4,n} \goto 0~~~\text{as}~~n\goto \infty\,. \label{conv:cond-c2-final-1}
 \end{align}
 We now combine \eqref{esti:c1n-c3n} and \eqref{conv:cond-c2-final-1} in \eqref{inq:cond-c2-1} to have
\begin{align}
\lim_{n\goto \infty} \sup_{t\in [0,T]}\| \bar{u}_n(t)-\bar{u}_{\bar{h}}(t)\|_{L^2(\R^d)}^2 \le C\sqrt{\delta}\,,\label{esti:final-C2}
\end{align}
where $C>0$ is a constant, independent of $\delta>0$. Thus, sending $\delta \goto 0$ in \eqref{esti:final-C2}, we arrive at the assertion \eqref{cond:c2-sufficient}. 

\section{Quadratic TCI inequality~(hence measure concentration) }\label{sec:proof-tci}
In this section, we prove Theorem \ref{thm:maintci}. Let $\mu$ be the law of the solution $u$ of equation \eqref{eq:spde} on the space $C([0,T]; L^2(\R^d))$ and $\nu$ be a probability measure on $C([0,T]; L^2(\R^d))$ such that $\nu  \ll  \mu$. Consider the new probability measure $\mathbb{P}^*$ on the filtered probability space $\big(\Omega, \mathbb{P},\mathcal{F}, \{\mathcal{F}_t\}\big)$ given by
$$ d \mathbb{P}_*= \frac{d\nu}{d\mu}\,d\mathbb{P}.$$
Then the Radon-Nikodyn derivative restricted on $\mathcal{F}_t$ given by
$\textbf{M}_t := \frac{{\rm d}\mathbb{P}_{*}}{{\rm d}\mathbb{P}}\vert_{\mathcal{F}_t},~t\in [0,T]$ forms a.s., continuous martingale with respect to the given probability measure $\mathbb{P}$. Then by Girsanov's theorem, there exists an $\mathcal{H}$-valued $\{\mathcal{F}_t\}$-adapted process $g(s)$
with the property $\mathbb{P}_{*} - a.s., \displaystyle  \int_0^t \|g(s)\|_{\mathcal{H}}^2\,{\rm d}s\, < \, \infty $ for all $t \in [0,T]$ such that
the stochastic process ${W}_*(\cdot)$, defined by
\[{W}_*(t) :=  W(t) - \int_0^t g(s)\,{\rm d}s\]
becomes a cylindrical Wiener process with respect to the new probability measure $\mathbb{P}_{*}$.
Moreover, thanks to \cite[Lemma $3.1$]{sarantsev-2019}, the martingale $\textbf{M}_t $ can be expressed as
\[\textbf{M}_t  =  \exp\bigg (\int_0^t \big\langle g(s),\,dW(s)\big\rangle - \frac{1}{2}\int_0^t \|g(s)\|_{\mathcal{H}}^2\,{\rm d}s\bigg), \quad \mathbb{P}_{*}- \text{a.s}.\] 
and the relative entropy (Kullback information) of $\nu$ with respect to $\mu$ can be given in terms of the process $g$:
\begin{equation}\label{eq:kullback-expression}
   \mathcal{H}(\nu|\mu) =  \frac{1}{2}\mathbb{E}_*\Big[\int_0^T \|g(s)\|_{\mathcal{H}}^2\,{\rm d}s\Big], 
\end{equation}
where $\mathbb{E}_*$ stands for the expectation under the  new probability measure $\mathbb{P}_{*}$.
\subsection{Proof of Theorem \ref{thm:maintci}}

For the above mentioned adapted process $g$, consider the following stochastic PDE:  
\begin{equation}\label{eq:tci-1}
\begin{aligned}
&dv_g(t)-{\rm div}_x(|\nabla v_g|^{p-2}\nabla v_g)\,dt= \sigma(v_g)\,dW_*(t) + \sigma(v_g)g(t)\,dt,~~~(t,x)\in (0,T]\times \R^d, \\
& v_g(0,x)=u_0(x),~~~x\in \R^d\,.
\end{aligned}
\end{equation}
Thanks to Girsanov's theorem and \cite[Theorems $1.3$ $\&$ $1.4$]{schmitz-23}, equation \eqref{eq:tci-1} has a unique solution $v_g$. Again, by \cite[Theorems $1.3$ $\&$ $1.4$]{schmitz-23}, the SPDE
\begin{equation}\label{eq:tci-2}
\begin{aligned}
&dv(t)-{\rm div}_x(|\nabla v|^{p-2}\nabla v)\,dt= \sigma(v)\,dW_*(t),~~~(t,x)\in (0,T]\times \R^d, \\
& v(0,x)=u_0(x),~~~x\in \R^d\,.
\end{aligned}
\end{equation}
has a unique solution, say $v$. Then it follows that, under the new probability measure $\mathbb{P}_*$, the laws of $(v, v_g)$ forms a coupling $(\mu, \nu)$. Observe that, by using the definition of the Wasserstein distance, one has
\begin{equation}\label{inq:important-tci}
\begin{aligned}
{\rm W}_2(\nu, \mu)^2 \le \mathbb{E}_*\Big[ \sup_{t\in [0,T]} \|v_g(t)-v(t)\|_{L^2(\R^d)}^2\Big]\,.
\end{aligned}
\end{equation}
Thus in view of \eqref{inq:important-tci} and \eqref{eq:kullback-expression}, to show that $\mu \in T^{2}({\tt C})$, it is sufficient to prove the following inequality
\begin{align}
\mathbb{E}_*\Big[ \sup_{t\in [0,T]} \|v_g(t)-v(t)\|_{L^2(\R^d)}^2\Big] \le {\tt C}\, \mathbb{E}_* \Big[\|g\|_{L^2(0,T; \mathcal{H})}^2\Big]\,. \label{inq:tci-suff-cond}
\end{align}
Since $v_g$ resp. $v$ satisfies the SPDE \eqref{eq:tci-1} resp. \eqref{eq:tci-2}, $Y_g:=v_g-v$ satisfies the SPDE 
\begin{align*}
dY_g(t) - {\rm div}_x \Big( |\nabla v_g|^{p-2}\nabla v_g-|\nabla v|^{p-2}\nabla v\Big)\,dt= \big(\sigma(v_g)-\sigma(v)\big)\,dW_*(t) + \sigma(v_g)g(t)\,dt
\end{align*}
with zero initial condition.  We use It\^{o} formula to $x\mapsto \|x\|_{L^2(\R^d)}^2$, the Cauchy-Schwartz inequality, Young's inequality, the assumptions \ref{A2}-\ref{A3}, and the inequality \eqref{inq:monotone} to have 
\begin{align}
\|Y_g(t)\|_{L^2(\R^d)}^2 & \le (1 + c_\sigma^2) \int_0^t \|Y_g(s)\|_{L^2(\R^d)}^2\,ds + \int_0^t \|\sigma(v_g)\|_{\mathcal{L}_2(\mathcal{H}, L^2(\R^d))}^2\|g(s)\|_{\mathcal{H}}^2\,ds \notag \\
 & \quad + 2 \sup_{s\in [0,t]} \Big| \int_0^s \langle Y_g(r), \big(\sigma(v_g(r))-\sigma(v(r))\big)\,dW_*(r)\rangle\Big| \notag \\
 & \le (1 + c_\sigma^2) \int_0^t \|Y_g(s)\|_{L^2(\R^d)}^2\,ds +  \bar{\sigma}_b^2\int_0^t \|g(s)\|_{\mathcal{H}}^2\,ds \notag \\
 & \quad + 2 \sup_{s\in [0,t]} \Big| \int_0^s \langle Y_g(r), \big(\sigma(v_g(r))-\sigma(v(r))\big)\,dW_*(r)\rangle\Big|\,, \label{inq:tci-prof-1}
\end{align}
where in the last inequality, we have used  the boundedness assumption of $\sigma$ i.e., \eqref{cond:bound-sigma}. Furthermore, the BDG inequality, Young's inequality and \eqref{inq:lip-sigma} yields
\begin{align}
& 2 \mathbb{E}_*\Big[  \sup_{s\in [0,t]} \Big| \int_0^s \langle Y_g(r), \big(\sigma(v_g(r))-\sigma(v(r))\big)\,dW_*(r)\rangle\Big|] \notag \\
& \le 8 \mathbb{E}_*\Big[ \Big( \int_0^t \|Y_g(r)\|_{L^2(\R^d)}^2 \|\sigma(v_g(r))-\sigma(v(r))\|_{\mathcal{L}_2(\mathcal{H}, L^2(\R^d))}^2\,dr \Big)^\frac{1}{2}\Big] \notag \\
& \le \frac{1}{2} \mathbb{E}_*\Big[ \sup_{s\in [0,t]}\|Y_g(s)\|_{L^2(\R^d)}^2\Big] + 32 \int_0^t \mathbb{E}_*\Big[ \|\sigma(v_g(r))-\sigma(v(r))\|_{\mathcal{L}_2(\mathcal{H}, L^2(\R^d))}^2\,dr \Big] \notag \\
& \le \frac{1}{2} \mathbb{E}_*\Big[ \sup_{s\in [0,t]}\|Y_g(s)\|_{L^2(\R^d)}^2\Big] + 32 c_\sigma^2 \int_0^t
\mathbb{E}_*\Big[\sup_{r\in [0,r]}\|Y_g(r)\|_{L^2(\R^d)}^2\Big]\,dr\,. \label{inq:tci-prof-2}
\end{align}
We combine \eqref{inq:tci-prof-1} and \eqref{inq:tci-prof-2}, and obtain 
\begin{align*}
\mathbb{E}_*\Big[ \sup_{s\in [0,t]}\|Y_g(s)\|_{L^2(\R^d)}^2\Big] \le 2\big(1 + 33\,c_\sigma^2\big)
\int_0^t \mathbb{E}_*\Big[ \sup_{r\in [0,s]}\|Y_g(r)\|_{L^2(\R^d)}^2\,ds\Big] + 2\bar{\sigma}_b^2 \mathbb{E}_*\Big[ \int_0^t 
\|g(s)\|_{\mathcal{H}}^2\,ds\Big]\,.
\end{align*}
An application of Gronwall's lemma then shows that
\begin{align*}
\mathbb{E}_*\Big[ \sup_{s\in [0,T]}\|Y_g(s)\|_{L^2(\R^d)}^2\Big] \le 2\bar{\sigma}_b^2 \exp\big\{2T\big(1 + 33\,c_\sigma^2\big)\big\} \mathbb{E}_*\Big[ \|g\|_{L^2(0,T;\mathcal{H})}^2\Big]\,.
\end{align*}
Thus \eqref{inq:tci-suff-cond} holds for the constant ${\tt C}:=2\bar{\sigma}_b^2 \exp\big\{2T\big(1 + 33\,c_\sigma^2\big)\big\} >0$.
In other words, the law $\mu$ of the solution of \eqref{eq:spde} satisfies the quadratic TCI inequality on $ C([0,T];L^2(\R^d))$ i.e., $\mu \in T^{2}({\tt C})$, where ${\tt C}:=2\bar{\sigma}_b^2 \exp\big\{2T\big(1 + 33\,c_\sigma^2\big)\big\} $. In particular, $\mu$  concentrates as a Borel measure on $C([0,T];L^2(\R^d))$. This completes the proof of Theorem \ref{thm:maintci}. 
\subsection{Declarations} The author would like to make the following declaration statements.

\begin{itemize}
\item {\bf Funding:}  This declaration is "not applicable".
\item{\bf Ethical Approval:} This declaration is "not applicable".

\item {\bf Availability of data and materials:}  Data sharing is not applicable to this article as no datasets were generated or analyzed during the current study.
\item{\bf Conflict of interest:} The author has not disclosed any competing interests.
\end{itemize}

\end{document}